\documentclass{article}
\usepackage[utf8]{inputenc} 
\usepackage[T1]{fontenc}    
\usepackage[a4paper, total={6in, 8in}]{geometry}
\usepackage{amsmath,amsthm}
\usepackage[hyphens]{url}
\usepackage{hyperref}       
\hypersetup{colorlinks, pdfencoding=auto, psdextra}
\makeatletter
\newcommand\blx@noerroretextools\relax
\makeatother
\usepackage[style=authoryear,mincitenames=1,maxcitenames=1,maxbibnames=9,backend=biber]{biblatex}
\usepackage{authblk}
\usepackage{booktabs}       
\usepackage{amsfonts}       
\usepackage{nicefrac}       
\usepackage{microtype}      
\usepackage{xcolor}         
\usepackage{parskip}

\usepackage{amsmath,amsthm}
\usepackage{makecell}
\usepackage{nicefrac}
\usepackage{enumitem}
\usepackage[ruled]{algorithm2e}
\setlist{nosep,leftmargin=*}

\usepackage{mdl}
\usepackage{autonum}

\newcommand{\parahead}[1]{{\bfseries #1.}}
\allowdisplaybreaks[4]

\title{Demographic Parity Constrained Minimax Optimal Regression under Linear Model}

\author[1,3]{Kazuto Fukuchi\thanks{fukuchi@cs.tsukuba.ac.jp}}
\author[2,3]{Jun Sakuma\thanks{sakuma@c.titech.ac.jp}}
\affil[1]{University of Tsukuba, Japan}
\affil[2]{Tokyo Institute of Technology}
\affil[3]{RIKEN AIP, Japan}

\begin{document}

\maketitle

\begin{abstract}
  We explore the minimax optimal error associated with a demographic parity-constrained regression problem within the context of a linear model. Our proposed model encompasses a broader range of discriminatory bias sources compared to the model presented by \textcite{Chzhen2022ARegression}. Our analysis reveals that the minimax optimal error for the demographic parity-constrained regression problem under our model is characterized by $\Theta(\nicefrac{dM}{n})$, where $n$ denotes the sample size, $d$ represents the dimensionality, and $M$ signifies the number of demographic groups arising from sensitive attributes. Moreover, we demonstrate that the minimax error increases in conjunction with a larger bias present in the model.
\end{abstract}

\section{Introduction}

Machine learning techniques have been incorporated into numerous automated decision-making systems, spanning critical domains such as employment, credit assessment, insurance, and security. Nevertheless, these systems can exhibit discriminatory behavior towards specific demographic groups, including gender, race, and ethnicity, potentially causing significant societal ramifications. This issue, known as the fairness problem, has attracted substantial attention within the machine learning research community. The growing focus on the fairness problem primarily arises from reported instances of unfair behavior in real-world systems, encompassing recidivism risk prediction~\autocite{Angwin2016MachineBias}, hiring practices~\autocite{Dastin2018AmazonWomen}, facial recognition~\autocite{Crockford2020HowRacist,Najibi2020RacialTechnology}, and credit scoring~\autocite{Vigdor2019AppleComplaints}.

Motivated by these concerns, a considerable body of research has explored regression problems subject to fairness constraints~\autocite{Komiyama2018NonconvexConstraints, Xie2017ControllableLearning, Moyer2018InvariantTraining, Agarwal2019FairAlgorithms, Chzhen2020FairBarycenters, Chzhen2020FairRecalibration, Mary2019Fairness-AwareTreatments, Narasimhan2020PairwiseRegression}. Numerous regression algorithms incorporating various fairness constraints have been developed to accommodate diverse contexts, with demographic parity~\autocite{Pedreshi2008Discrimination-awareMining} and equalized odds~\autocite{Hardt2016EqualityLearning} being the predominant fairness constraints adopted by these methods.

In this study, we focus on the regression problem under the fairness constraint of demographic parity~\autocite{Pedreshi2008Discrimination-awareMining}. Existing literature primarily concentrates on the development of fair regression algorithms, and their performance evaluation predominantly relies on empirical analyses. Such evaluations, however, only offer performance guarantees for specific scenarios explored in the experiments, which may result in poor performance in unexamined situations. To ensure the algorithm's robust performance across a wider range of contexts and obtain a comprehensive understanding of the fair regression problem, a theoretical analysis of statistical efficiency is indispensable.

Several studies have introduced fair regression algorithms accompanied by theoretical analyses of their statistical efficiency in terms of accuracy and fairness. \textcite{Agarwal2019FairAlgorithms} designed a demographic parity-based fair regression algorithm using reduction methods~\autocite{Agarwal2018AClassification} and established upper bounds on its empirical approximation errors for accuracy and fairness using Rademacher complexity. \textcite{Chzhen2020FairRecalibration} proposed a discretization-based fair regression algorithm, deriving upper bounds on the mean squared excess risk for accuracy and a Kolmogorov distance-based score for demographic parity as fairness guarantees. \textcite{Chzhen2020FairBarycenters} derived the Bayes optimal regressor under a demographic parity constraint, providing upper bounds on the mean absolute deviation from the Bayes optimal regressor for accuracy and a Kolmogorov distance-based score for demographic parity as fairness guarantees. Despite ensuring low error and fair treatment even in non-linear models, it remains unclear if these guarantees represent optimal performance among possible algorithms.

\begin{table}
  \centering
  \caption{Comparison between \textcite{Chzhen2022ARegression}'s and our models. Each checkmark signifies the presence of an influence exerted by the sensitive attribute on the respective variable. }\label{tbl:comp-models}
  \begin{tabular}{l|lll}
    \toprule
       & \makecell{partial\\coefficients} & intercept & \makecell{non-sensitive\\features} \\
    \midrule
      \textcite{Chzhen2022ARegression} &  & \checkmark & \\
      ours & \checkmark & \checkmark & \checkmark \\
    \bottomrule
  \end{tabular}
\end{table}
\parahead{Minimax optimal fair regression}
Numerous researchers have investigated minimax optimal regression algorithms, the best possible algorithm, without addressing fairness considerations~\autocite{Stone1980OptimalEstimators,Tsybakov2003OptimalAggregation,Oliveira2016TheSquares,Mourtada2022ExactMatrices}. In contrast to standard regression problems, minimax optimality in fair regression problems remains relatively unexplored, with a notable exception being the recent work by \textcite{Chzhen2022ARegression}. They examine minimax optimality in fair regression problems, incorporating demographic parity constraints within the following linear model:
\begin{align}
Y = \abrace*{\beta^*, X} + b_S + \xi \where X \sim N(0,\Sigma). \label{eq:model-chzhen}
\end{align}
In this model, $Y$, $X$, and $S$ represent the outcome, non-sensitive features, and a sensitive attribute, respectively. $\abrace{\cdot,\cdot}$ denotes the inner product, $\xi$ represents zero-mean noise, and $\Sigma$ is an arbitrary covariance matrix. For example, in salary calculations, $X$ and $S$ correspond to working hours and gender, respectively, with $b_S$ and $\beta^*$ signifying the base salary and hourly wage. In \cref{eq:model-chzhen}, the salary $Y$ is determined by the base salary $b_S$ and the product of working hours $X$ and an hourly wage $\beta^*$.

The model in \cref{eq:model-chzhen} exhibits limitations pertaining to its applicability across various scenarios. We elucidate these limitations by discussing the notion of {\em direct discrimination} and {\em indirect discrimination}, summarized succinctly in the second row of \cref{tbl:comp-models}. Direct discrimination occurs when the sensitive attribute influences the outcome, regardless of non-sensitive features. The model in \cref{eq:model-chzhen} can treat direct discrimination resulting from the dependency of the intercept $b_S$ on $S$; for example, it can capture discrimination due to basing base salary on gender~(third column in \cref{tbl:comp-models}). However, it is imperative to underscore that the model in \cref{eq:model-chzhen} fails to handle direct discrimination arising from the partial~(regression) coefficients $\beta^*$, as these are independent of $S$; for instance, it cannot accommodate discrimination due to gender-dependent hourly wages~(second column in \cref{tbl:comp-models}).

Indirect discrimination~(or redlining effect~\autocite{Calders2010ThreeClassification}) constitutes another source of unfair bias, arising when the sensitive attribute influences the outcome through its correlation with non-sensitive features. The presence of the dependency between non-sensitive features and the sensitive attribute signifies indirect discrimination. In the model in \cref{eq:model-chzhen}, non-sensitive features $X$ is independet from the sensitive attribute $S$, thereby implying an absence of indirect discrimination~(forth column in \cref{tbl:comp-models}).

\textcite{Chzhen2022ARegression} effectively revealed the minimax optimal error for fair regression problems involving direct discrimination due to varying intercepts associated with sensitive attributes. However, their research does not address direct discrimination from partial coefficients and indirect discrimination through non-sensitive features.

\parahead{Our model and contributions} 
In this study, we investigate the minimax optimality of the fair regression problem in the context of the following model:
\begin{align}
Y = \abrace*{\beta^*_S, X} + \xi \where X \sim N(\mu_S,\sigma_XI), \label{eq:model}
\end{align}
where $\sigma_X > 0$, and $I$ denotes the identity matrix. The subscript in $\beta^*_S$ and $\mu_S$ signifies that our model varies regression coefficients and the mean of non-sensitive features based on the sensitive attribute.

Compared to the model proposed by \textcite{Chzhen2022ARegression}, our model accommodates a broader range of direct and indirect discrimination. These discriminations can be characterized as follows:
\begin{itemize}
    \item (Direct discrimination) Our model accommodates direct discrimination through discrepancies in $\beta^*_S$ concerning $S$, as the regression coefficients $\beta^*_S$ hinge on the sensitive attribute $S$~(second and third columns on the third row in \cref{tbl:comp-models}). This includes, for instance, discrimination arising from varying base salaries and hourly wages. Divergent partial coefficients yield varied outcome variance amongst $S$, while disparate intercepts relative to $S$ merely alter the outcome's mean. Hence, our model introduces an additional challenge of attenuating direct discrimination through disparate variance, alongside mitigating direct discrimination through disparate mean. This presents a stark contrast to \textcite{Chzhen2022ARegression}'s model, which solely focuses on mitigating discrimination via the mean without considering the variance.
    \item (Indirect discrimination) The sensitive attribute $S$ affects the mean of non-sensitive features $X$, as denoted by the subscript of $\mu_S$. Our model thereby introduces indirect discrimination through variations in $\mu_S$ with respect to $S$~(e.g., disparate working hours by gender). To alleviate this form of indirect discrimination, $\mu_S$ needs to be estimated to adjust the learned regressor, thereby ensuring its output remains invariant to differing $\mu_S$. Therefore, our model presents an additional complexity in estimating $\mu_S$ for mitigating indirect discrimination.
\end{itemize}
Overall, our model demonstrates an expanded dependency of partial coefficients~(direct discrimination) and non-sensitive features~(indirect discrimination) on the sensitive attribute~(second and fourth columns of \cref{tbl:comp-models}). 

The principal contribution of this paper lies in the establishment of matching upper and lower bounds on the minimax optimal error~(i.e., the error corresponding to the minimax optimal regression algorithm) and the proposition of a regression algorithm that achieves this optimal error under \cref{eq:model}. The optimal error elucidates several insights:
\begin{itemize}
\item (Direct discrimination) The optimal error comprises a term reflecting the outcome's variance heterogeneity but excludes that of the outcome's mean. This insight implies that mitigating direct discrimination due to the outcome's variance sacrifices statistical efficiency, whereas addressing direct discrimination due to the outcome's mean does not entail this cost. This term effectively quantifies the cost of mitigating direct discrimination in variance and is absent from the optimal error of the \textcite{Chzhen2022ARegression}'s model. Its identification, thus, signifies a crucial contribution of our research.
\item (Indirect discrimination) Our lower bound is independent of the term associated with indirect discrimination. Although this evidence is not definitive, it hints at the potential for mitigating indirect discrimination without additional costs under certain conditions. This observation sets the stage for future research focused on developing cost-effective strategies to tackle indirect discrimination.
\end{itemize}
Our technical contributions to establish these bounds are detailed in \cref{sec:tech-diff}.

\parahead{Notations}
Given a positive integer $m$, define $[m] = \cbrace{1, ..., m}$. For a finite set $A$, denote its cardinality by $\abs{A}$. Given an event $\event$, its complement is represented as $\event^c$, and its probability is denoted by $\p\cbrace{\event}$. For a random variable $X$, its expectation is $\Mean[X]$, and its associated sigma-algebra is $\sigma(X)$. For two real values $a$ and $b$, the notations $a \lor b = \max\cbrace{a, b}$ and $a \land b = \min\cbrace{a, b}$ are used. For a square matrix $A \in \RealSet^{d \times d}$, its maximum and minimum eigenvalues are denoted by $\lambda_{\max}(A)$ and $\lambda_{\min}(A)$, respectively, and its transpose is represented by $A^{\top}$. The set of unit vectors is given by $\mathbb{S}_{d-1}$. For a sequence $a_t$ indexed by $t \in \dom{T}$, the notation $a_\cdot$ denotes the sequence $(a_t)_{t \in \dom{T}}$.

\section{Problem Setup}

\subsection{Model and Learning Algorithm}
\parahead{Model} The proposed model, described in the introduction, is formulated according to \cref{eq:model}. We consider $X \in \mathbb{R}^d$ and $S \in [M]$ where $M \ge 2$. The noise variable, $\xi$, is assumed to follow a Gaussian distribution with zero mean and variance $\sigma^2_\xi > 0$. We define $p_s = \p\cbrace{S=s}$ for all $s \in [M]$, and the optimal regression function is denoted as $f^*(x,s) = \abrace{\beta^*_s, x}$.
 
\parahead{Learning algorithm} \sloppy Given $n$ i.i.d. copies of the tuple $(X, S, Y)$, denoted as $D_n = \cbrace{(X_1, S_1, Y_1), ..., (X_n, S_n, Y_n)}$, the goal is to construct a regression function $f$ that maps $(X, S)$ to $Y$, represented as $\hat{f}_n$. The learner seeks to optimize the accuracy of $\hat{f}_n$ while satisfying a fairness constraint. The definitions of fairness and accuracy are provided in subsequent subsections.

\subsection{Fairness}
\parahead{Demographic parity} We utilize demographic parity~\autocite{Pedreshi2008Discrimination-awareMining} as our fairness criterion. A regressor $f$ adheres to demographic parity if its output distribution is invariant when conditioned on $S=s$.
\begin{definition}\label{def:dp}
A regressor $f$ satisfies (strong) demographic parity if, for all $s,s' \in [M]$, and for all $E \in \sigma(f(X,S))$, $\p\cbrace*{f(X,S) \in E | S = s} = \p\cbrace*{f(X,S) \in E | S = s'}$.
\end{definition}
Denote the set of all regressors fulfilling demographic parity for a given distribution of $X$, parameterized by $\mu_\cdot$, as $\fset{F}_{\mathrm{DP}}(\mu\cdot)$.

\parahead{Fairness consistency} Instead of enforcing strict demographic parity (\cref{def:dp}), which results in the regressor to be a constant function due to the unknown $(X, S)$ distribution, we introduce \emph{fairness consistency}~(\cref{def:consistent}). This concept demands the learned regressor to converge to a fair regressor as the sample size $n$ approaches infinity.

To define ``convergence'', we introduce the \textit{unfairness score} $U(f) \ge 0$, where a lower $U(f)$ indicates a higher fairness level. $U(f) = 0$ if and only if $f$ achieves demographic parity~(\cref{def:dp}). We claim the learned regressor $\hat{f}_n$ converges to an exactly fair regressor when $U(\hat{f}_n) \to 0$ as $n \to \infty$. \begin{definition}\label{def:consistent}
A learning algorithm is $(\alpha, \delta)$-consistently fair for an unfairness score $U$ if there exist constants $n_0 \ge 0$ and $C > 0$, independent of $n$, such that $\p\cbrace{U(\hat{f}_n) > Cn^{-\alpha}} \le \delta$ for all $n \ge n_0$, with randomness arising from the training sample via $\hat{f}_n$.
\end{definition}
Note that an $(\alpha,\delta)$-consistently fair regressor $\hat{f}_n$ exhibits $(\alpha',\delta)$-consistent fairness for any $\alpha' \in (0,\alpha]$.

We adopt a specific unfairness score using the Wasserstein distance. Given two probability measures $\nu$ and $\nu'$ over $\RealSet$, $\Pi(\nu,\nu')$ denotes the set of all coupling measures $\pi$ satisfying $\pi(A\times\RealSet)=\nu(A)$ and $\pi(\RealSet\times A')=\nu'(A')$ for every measurable sets $A,A' \subset \RealSet$. The 2-Wasserstein distance $W_2$ between $\nu$ and $\nu'$ is expressed as $W^2_2(\nu,\nu') = \inf_{\pi \in \Pi(\nu,\nu')}\int\paren{z - z'}^2\pi(dz,dz')$. Our unfairness score is then formulated as:
\begin{align}
    U(f) = \max_{s,s'\in[M]}W_2(\nu_{f|s}, \nu_{f|s'}),
\end{align}
where $\nu_{f|s}$ represents the distribution of $f(X,S)$ conditioned on $S=s$. Prior works, including \autocite{Agarwal2019FairAlgorithms,Chzhen2020FairBarycenters,Chzhen2020FairRecalibration,Chzhen2022ARegression}, have adopted different unfairness scores~(see the appendix for details).

\subsection{Accuracy}
Under the fairness consistency constraint, the learner's objective is to obtain a fair approximation of $f^*$, denoted as $f^*_{\mathrm{DP}}$, which is the closest regressor to $f^*$ within $\mathcal{F}_{\mathrm{DP}}(\mu_\cdot)$ using the $L^2$ distance:
\begin{align}
    f^*_{\mathrm{DP}} = \argmin_{f \in \mathcal{F}_{\mathrm{DP}}(\mu_\cdot)} \Mean\bracket*{\paren*{f(X,S) - f^*(X,S)}^2}.
\end{align}
To evaluate the inaccuracy of a regressor $f$, we compute the mean squared deviation from $f^*_{\mathrm{DP}}$:
\begin{align}
    \mathcal{E}(f;\beta^*_\cdot,\mu_\cdot) = \Mean\bracket*{\paren*{f(X,S) - f^*_{\mathrm{DP}}(X,S)}^2}. \label{eq:deviate}
\end{align}
\textcite{Chzhen2020FairBarycenters,Chzhen2020FairRecalibration} employ similar definitions, differing only in the choice of deviation metric.

This paper aims to identify the minimax optimal regression algorithm, which minimizes \cref{eq:deviate} while maintaining fairness consistency. Given parameters $\alpha > 0$ and $\delta \in (0,1)$, the optimal error is formulated as:
\begin{align}
    \mathcal{E}_{n}(\alpha,\delta) = \inf_{\hat{f}_n:(\alpha,\delta)\text{-consistently fair}}\sup_{\beta^*_\cdot \in \mathcal{B}, \mu_\cdot \in \mathcal{M}}\Mean\bracket{\mathcal{E}(\hat{f}_n;\beta^*_\cdot,\mu_\cdot)}, \label{eq:minimax}
\end{align}
where the infimum is taken over all $(\alpha,\delta)$-consistently fair algorithms, and $\mathcal{B}$ and $\mathcal{M}$ represent the sets of possible $\beta^*_\cdot$ and $\mu_\cdot$, respectively.

\section{Main Results}
Our main result is to establish the minimax optimal error bound, delineating the dependency on the diversity of conditional outcome variances concerning the sensitive attribute. This diversity of the variances is quantified via a parameter $B > 0$, which is defined such that it satisfies:
\begin{align}
\max_s\norm{\beta^*_s} \le B \textand \frac{\paren{\sum_sp_{s}\norm{\beta^*_{s}}}^2}{M}\sum_s\frac{1}{\norm{\beta^*_{s}}^{2}} \le B^2. \label{eq:rist-beta}
\end{align}
The left-hand side of the second inequality in \cref{eq:rist-beta} forms as a product of two factors: the weighted average norms, $\paren{\sum_sp_{s}\norm{\beta^*_{s}}}^2$, and the averaged inverse norms, $\frac{1}{M}\sum_s\frac{1}{\norm{\beta^*_{s}}^{2}}$. As the norms increase, the first factor (weighted average norms) has the propensity to grow, while the second factor (averaged inverse norms) tends to rise when the norms decrease. Maximizing the product of these two elements involves a delicate balancing act: the norms of some groups need to be large, while the norms of other groups need to be smaller. As such, the left-hand side of the second inequality in \cref{eq:rist-beta} can increase when the norms $\norm{\beta^*_s}$ display diversity.

We adopt mild assumptions on $\beta^*\cdot$ and $\mu\cdot$. Let $\dom{B}$ denote the set of $\beta_\cdot$ satisfying \cref{eq:rist-beta}. Assume there exists a finite universal constant $U > 0$ such that $\norm{\mu_s} \le U$ for all $s \in [M]$, leading to $\mathcal{M} = \cbrace{\mu_\cdot \in \RealSet^{d\times M} : \forall s \in [M], \norm{\mu_s} \le U}$. Our analysis relies on these assumptions.

Our main results are as follows:
\begin{theorem}\label{thm:main-err}
Given $\alpha \in (0,\nicefrac{1}{2}]$ and $\delta \in (0,1)$, suppose $M(d-1) > 16$ and $n \ge 12(3d\lor4\ln(M/\delta))/\min_{s\in[M]}p_s$. Then, there exist universal constants $C > 0$ and $c > 0$ such that
\begin{align}
c\frac{\sigma^2_\xi B^2 dM}{n} - o\paren*{\frac{1}{n}} \le \mathcal{E}_{n}(\alpha,\delta) \le C\frac{\sigma^2_\xi B^2 dM \lor \sigma^2_XB^2M \lor B^2U^2}{n} + o\paren*{\frac{1}{n}}.
\end{align}
\end{theorem}
\Cref{thm:main-err} illustrates that the optimal error is $\nicefrac{\sigma^2\xi B^2 dM}{n}$ up to a constant factor which may potentially depend on $\sigma_X$ and $U$.
The implications of \cref{thm:main-err} can be summarized as follows: 
\begin{enumerate}
    \item The optimal error for the standard linear regression problem can be denoted as $\nicefrac{d}{n}$~\autocite{Mourtada2022ExactMatrices}. The dependency on $n$ and $d$ is consistent with the standard case, provided $\alpha \in (0, \nicefrac{1}{2}]$.
    \item The term $dM$ denotes the number of unknown parameters in \cref{eq:model}, comprising $\beta^*_1,..,\beta^*_M \in \RealSet^d$ and $\mu_1,...,\mu_M \in \RealSet^d$. This dependency on the number of unknown parameters is a common characteristics observed in statistical estimation problems.
    \item {\bfseries (Direct discrimination)} The minimax error delineated in \cref{thm:main-err} demonstrates a dependency on parameter $B$. As the variation of $\norm{\beta^*_s}$ with respect to $s$ increases, so does the magnitude of $B$. Hence, $B$ serves as a measure of the difficulty in mitigating direct discrimination due to the outcome's variance. This unique quantification of difficulty is absent in standard regression problems and specific to fair regression problems.
    \item {\bfseries (Indirect discrimination)} The lower bound precludes parameters associated with indirect discrimination. It is conceivable that biases arising from indirect discrimination can be reduced without extra costs, provided the dependence of $X$ on $S$ exists only in its mean. Investigating and clarifying this aspect offers a promising direction for future research.
    \item The minimax error is invariant to $\alpha$ and $\delta$, implying that the learning process does not introduce unfair bias for $\alpha \in (0,\nicefrac{1}{2}]$. However, the case for $\alpha \ge \nicefrac{1}{2}$ remains unexplored and poses a significant research challenge.
    \item The gap between the upper and lower bounds regarding $\sigma_X$ and $U$ remains, making narrowing this gap an essential future research direction.
\end{enumerate}
\begin{remark}\label{rmk:err-chz}
Direct comparison of the minimax error between our model and that of \cref{eq:model-chzhen} is not feasible due to the differing $f^*_{\mathrm{DP}}$ across the models. However, the emergence of the fairness-specific term $B$ can be unequivocally identified as a novel contribution in our study. Notably, the minimax error validated by \textcite{Chzhen2022ARegression} is congruent with the minimax optimal error of standard linear regression within their model, a contrast to our findings.
\end{remark}

To prove \cref{thm:main-err}, we initiate by constructing the estimator detailed in \cref{sec:est}. We then prove in \cref{sec:upper} that the estimator satisfies 1) $(\alpha,\delta)$-fairness consistency for $\alpha \in (0,\nicefrac{1}{2}]$, and 2) the error aligns with the upper bound specified in \cref{thm:main-err}. Subsequently, we present a sketch of the proof for the lower bound in \cref{thm:main-err} in \cref{sec:lower}. All omitted proofs can be found in the appendices.

\section{Technical Difficulties in Minimax Optimality Analyses}\label{sec:tech-diff}
In this section, we expound on the challenges arising from the analysis of minimax optimality for our problem. First, we introduce the closed-form expression for the Bayes optimal fair regressor $f^*_{\mathrm{DP}}$. We then outline the technical difficulties encountered during the analysis.

\parahead{Bayes optimal fair regressor under \cref{eq:model}}
\textcite{Chzhen2020FairBarycenters} present a characterization of regression error and the corresponding regressor minimizing the mean squared error under the demographic parity constraint. Building upon the results from \textcite{Chzhen2020FairBarycenters}, we derive the closed-form expression for $f^*_{\mathrm{DP}}$ in the following lemma.
\begin{lemma}\label{lem:bayes-opt-linear}
Given the model in \cref{eq:model}, the Bayes optimal regressor adhering to the demographic parity constraint can be formulated as
\begin{align}
    f^*_{\mathrm{DP}}(x,s) = \overline{\norm{\beta^*_\cdot}}\abrace*{\frac{\beta^*_s}{\norm{\beta^*_s}},x - \mu_s} + \smashoperator{\sum{s'\in[M]}}p_{s'}\abrace*{\beta^*_{s'},\mu_{s'}}, \label{eq:opt-cls}
\end{align}
where $\overline{\norm{\beta^*_\cdot}} = \sum_{s \in [M]}p_s\norm{\beta^*_s}$.
\end{lemma}

\parahead{Technical difficulty in deriving the upper bound in \cref{thm:main-err}} 
To obtain the upper bound in \cref{thm:main-err}, we first construct an estimator for the regression function in \cref{eq:opt-cls} and analyze its regression error. This entails developing estimators for individual components in \cref{eq:opt-cls}~(e.g., $\overline{\norm{\beta^*_\cdot}}$, $\nicefrac{\beta^*_s}{\norm{\beta^*_s}}$, $\mu_s$, etc.) and substituting them into \cref{eq:opt-cls}. The upper bound on $\mathcal{E}_n(\alpha,\delta)$ is derived by combining estimation error bounds for each component's estimator. However, to our best knowledge, no existing estimators provide bounds for the norm~($\norm{\beta^*_\cdot}$) and direction~($\nicefrac{\beta^*_s}{\norm{\beta^*_s}}$) of regression coefficients. A direct approach involves computing the norm and direction of the OLS estimator, but standard analyses for OLS do not yield bounds on the estimation errors.

The main challenge in deriving the upper bound of \cref{thm:main-err} lies in analyzing the following problem: given $X$ following a non-isotropic Gaussian distribution with mean $\mu$, find upper bounds on $\Mean\bracket{(X/\norm{X} - \mu/\norm{\mu})^2}$ and $\Mean\bracket{(\norm{X}-\norm{\mu})^2}$. Solving this problem provides estimation errors for the norm and direction estimators, as the OLS estimator is an unbiased estimator with noise following the non-isotropic Gaussian distribution. Our key technical contribution is the derivation of these bounds~(\cref{thm:err-norm,thm:err-unit-beta}).

\parahead{Technical difficulty in deriving the lower bound in \cref{thm:main-err}} 
The minimax optimal error characterizes the intrinsic complexity of the regression problem, as no algorithm can surpass this error. In our analysis of the lower bound presented in \cref{thm:main-err}, we demonstrate that the fair regression problem's complexity, under the model \cref{eq:model}, is characterized by the complexity in estimating the direction $\nicefrac{\beta^*_s}{\norm{\beta^*_s}}$. The primary challenge lies in establishing this characterization.

To overcome this challenge, we investigate the geometric structure of the error term $\mathcal{E}_n(f;\beta^*_\cdot,\mu_\cdot)$ concerning the parameters $\beta^*_\cdot$ and $\mu_\cdot$. We then reveal that the geometric structure of $\mathcal{E}_n(f;\beta^*_\cdot,\mu_\cdot)$ is characterized by the geometric structure of the direction $\nicefrac{\beta^*_s}{\norm{\beta^*_s}}$~(\cref{thm:two-point-lower}).

\section{Estimator}\label{sec:est}
In this section, we present a detailed construction of the estimators that attain the minimax error as delineated in \cref{thm:main-err}. Existing theoretical results, such as those found in \textcite{Agarwal2019FairAlgorithms, Chzhen2020FairBarycenters, Chzhen2020FairRecalibration}, are incapable of addressing unbounded non-sensitive features $X$ or unbounded outcomes $Y$, rendering them inapplicable to our problem. Consequently, we have developed a novel estimator accompanied by rigorous analytical techniques.

\parahead{Estimator construction}
In constructing the optimal regressor for model \cref{eq:model}, we leverage the results from \cref{lem:bayes-opt-linear} and employ a plugin estimator. The method involves estimating the components of terms in \cref{eq:opt-cls} and substituting the obtained estimates into the same equation. Concretely, we derive estimators $\widehat{\norm{\beta_\cdot}}$, $\tilde\beta_s$, $\hat\mu_s$, $\hat{p}_{s}$, $\hat\beta'_{s}$, and $\hat\mu'_{s}$, with the following correspondence:
\begin{align}
    \underbracket{\overline{\norm{\beta^*_\cdot}}}_{\widehat{\norm{\beta_\cdot}}}\abrace*{\underbracket{\frac{\beta^*_s}{\norm{\beta^*_s}}}_{\tilde\beta_s},x - \underbracket{\mu_s}_{\hat\mu_s}} + \sum_{s'\in[M]}\underbracket{p_{s'}}_{\hat{p}_{s'}}\abrace*{\underbracket{\beta^*_{s'}}_{\hat\beta'_{s'}},\underbracket{\mu_{s'}}_{\hat\mu'_{s'}}}.
\end{align}

\begin{table}[t]
    \centering
    \caption{Estimator construction. In this table, $\hat\beta_{b,s}$ and $\hat\beta'_{b,s}$ denote OLS estimands obtained from subsets $D_{b,s}$ and $D'_{b,s}$, respectively. ``Sample'' refers to the subset utilized for estimand calculation, while ``Definition'' provides the corresponding estimator's definition. ``Sample`` in $\widehat{\norm{\beta_\cdot}}$ is left empty, as it is derived from $\hat{p}_s$ and $\widehat{\norm{\beta_s}}$.  }
    \begin{tabular}{c|c|l}
    \toprule
        Estimator & Sample & Definition \\
    \midrule
        $\hat{p}_s$ & $n_\cdot$ & $\hat{p}_s = \nicefrac{n_s}{n}$ \\
        $\widehat{\norm{\beta_s}}$ & $D_{1,s}$ & $\widehat{\norm{\beta_s}} = \norm{\hat\beta_{1,s}}$ if $n_s > 18d$, and $\widehat{\norm{\beta_s}}=0$ otherwise \\
        $\widehat{\norm{\beta_\cdot}}$ & - & $\widehat{\norm{\beta_\cdot}} = \sum_{s \in [M]}\hat{p}_s\widehat{\norm{\beta_s}}$ \\
        $\tilde\beta_s$ & $D_{2,s}$ & $\tilde\beta_s = \hat\beta_{2,s}/\norm{\hat\beta_{2,s}}$ if $n_s > 18d$, and $\tilde\beta_s = 0$ otherwise \\
        $\hat\mu_s$ & $D_{3,s}$ & $\hat\mu_s = \frac{1}{n_{3,s}}\sum_{i=1}^{n_{3,s}}X_{3,s,i}$ \\
        $\hat\beta'_s$ & $D'_{1,s}$ & $\hat\beta'_s = \hat\beta'_{1,s}$ if $n_{s} > 12d$, and $\hat\beta'_s = 0$ otherwise \\
        $\hat\mu'_s$ & $D'_{2,s}$ & $\hat\mu'_s = \frac{1}{n'_{2,s}}\sum_{i=1}^{n'_{2,s}}X'_{2,s,i}$ \\
    \bottomrule
    \end{tabular}
    \label{tab:each-estimator}
\end{table}

For technical reasons, we partition the sample to calculate each estimand. Each estimator is assigned a corresponding subset, as shown in \cref{tab:each-estimator}. Under specific conditions, $n_s > 18d$ or $n_s > 12d$, estimators may exhibit altered behavior, primarily as technical considerations for subsequent analyses. We detail the partitioning process as follows. First, we create a histogram of the sensitive attribute $S_i$, denoted as $n_\cdot = (n_1,...,n_M)$, with $n_s = \abs{\cbrace{i \in [n] : S_i = s}}$. Simultaneously, we form group-wise samples $D_s = \cbrace{(X_i,Y_i) : i \in [n], S_i = s}$. For each $s \in [M]$, we partition $D_s$ into $D_{1,s}$, $D_{2,s}$, and $D_{3,s}$, ensuring $\abs{D_{b,s}} \coloneqq n_{b,s} \ge \floor{n_s/3}$ for $b \in [3]$. Using $n_\cdot$, $D_{1,s}$, $D_{2,s}$, and $D_{3,s}$, we estimate $\hat{p}_s$, $\widehat{\norm{\beta_s}}$, $\tilde\beta_s$, and $\hat\mu_s$, respectively. The combination of $\hat{p}_s$ and $\widehat{\norm{\beta_s}}$ yields $\widehat{\norm{\beta_\cdot}}$. Furthermore, we generate a duplicate of $D_s$, denoted as $D'_s$, and partition it into $D'_{1,s}$ and $D'_{2,s}$, satisfying $\abs{D'_{b,s}} \coloneqq n'_{b,s} \ge \floor{n_s/2}$ for $b \in [2]$. We then use $D'_{1,s}$ and $D'_{2,s}$ to estimate $\hat\beta'_s$ and $\hat\mu'_s$. Precise definitions of the estimator construction and subset partitioning can be found in the appendices.

Incorporating the derived estimators, we construct the final regressor as:
\begin{align}
    \hat{f}_n(x,s) = \widehat{\norm{\beta_{\cdot}}}\abrace*{\tilde\beta_s, x - \hat\mu_s} + \sum_{s'\in[M]}\hat{p}_{s'}\abrace*{\hat\beta'_s,\hat\mu'_s}.\label{eq:loose-reg}
\end{align}

\section{Upper Bound Analyses}\label{sec:upper}
In this section, we demonstrate the achievability of the upper bound presented in \cref{thm:main-err} utilizing the estimator delineated in \cref{sec:est}. Initially, we conduct an analysis of the estimator's fairness guarantee, subsequently progressing to an examination of the estimator's mean squared deviation.

\subsection{Analysis of Fairness}
For our fairness guarantee on $\hat{f}_n$, we demonstrate the following theorem.
\begin{theorem}\label{thm:consist-bound}
 If $n \ge 48\ln(M/\delta)/\min_sp_s$, we have for $\delta \in (0,1)$,
 \begin{align}
     \p\cbrace*{\max_{s,s'\in[M]}W_2\paren*{\nu_{\hat{f}_{n}|s},\nu_{\hat{f}_{n}|s'}} > 4B\sigma_X\sigma_X\sqrt{\frac{48\ln(M/\delta)}{\min_{s''\in[M]}np_{s''}}}} \le \delta.
 \end{align}
\end{theorem}
By proving \cref{thm:consist-bound}, we can immediately confirm that the estimator adheres to $(\alpha,\delta)$-fairness consistency with $\alpha \in (0,\nicefrac{1}{2}]$.

\subsection{Analysis of Estimation Error}\label{sec:est-err}
In this subsection, we derive an upper bound for the estimation error presented in \cref{thm:main-err}, focusing on the estimator introduced in \cref{sec:est}. To derive the upper bound in \cref{thm:main-err}, we begin by decomposing the mean squared deviation of the estimator in \cref{eq:loose-reg} as follows:
\begin{theorem}\label{thm:decompose}
 For the estimator defined in \cref{eq:loose-reg}, the mean square deviation  from $f^*_{\mathrm{DP}}$ is bounded above by
 \begin{multline}
      \sum_{s \in [M]}p_s\Mean\bracket[\Bigg]{\paren[\Bigg]{\Mean\bracket*{\widehat{\norm{\beta_\cdot}}^2\middle|n_\cdot}^{\nicefrac{1}{2}}\Mean\bracket*{\abrace*{\tilde\beta_s,\mu_s-\hat\mu_{s}}^2\middle|n_\cdot}^{\nicefrac{1}{2}} + \sigma_X\Mean\bracket*{\paren*{\widehat{\norm{\beta_\cdot}} - \overline{\norm*{\beta^*_\cdot}}}^2\middle|n_\cdot}^{\nicefrac{1}{2}} + \\ \sigma_X\overline{\norm{\beta^*_\cdot}}\Mean\bracket*{\norm*{\tilde\beta_s - {\beta^*_s}/{\norm{\beta^*_s}}}^2\middle|n_\cdot}^{\nicefrac{1}{2}} + \Mean\bracket*{\paren*{\sum_{s' \in [M]}\hat{p}_{s'}\abrace*{\hat\beta'_{s'} - \beta^*_{s'},\hat\mu'_{s'}}}^2\middle|n_\cdot}^{\nicefrac{1}{2}} + \\ \Mean\bracket*{\paren*{\sum_{s' \in [M]}\hat{p}_{s'}\abrace*{\beta^*_{s'},\hat\mu'_{s'} - \mu_{s'}}}^2\middle|n_\cdot}^{\nicefrac{1}{2}} + \abs*{\sum_{s' \in [M]}\paren*{\hat{p}_{s'} - p_{s'}}\abrace*{\beta^*_{s'},\mu_{s'}}}}^2}. \label{eq:decompose}
 \end{multline}
\end{theorem}
In \cref{eq:decompose}, the terms correspond to the estimation errors of $\hat\mu_s$, $\widehat{\norm{\beta_\cdot}}$, $\tilde\beta_s$, $\hat\beta'_s$, $\hat\mu'_s$, and $\hat{p}_s$, respectively. Standard techniques for the OLS estimator and empirical average yield upper bounds for the first, fourth, fifth, and sixth terms. Nevertheless, the second and third terms in \cref{eq:decompose} involve non-linear transformations of the OLS estimator (i.e., taking the norm or dividing by the norm), complicating their error analysis. This section's primary technical contributions involve establishing tight upper bounds for the second and third terms in \cref{eq:decompose}.

\parahead{Estimation error of norm and direction of $\beta^*_s$}
Consider $X_1,...,X_n \iidsim N(\mu,\sigma^2_XI)$, $\beta^* \in \RealSet^d$ with $\norm{\beta^*} \le B$ for some $B > 0$, and $\xi_1,...,\xi_n \iidsim N(0,\sigma^2_\xi)$. Define $Y_i = \abrace{\beta^*,X_i} + \xi_i$. The OLS estimator of $\beta^*$ is given by $\hat\beta = \paren{\frac{1}{n}X^{\top}X}^{-1}\paren{\frac{1}{n}X^{\top}Y}$, where $X = (X_1 \cdots X_n)^{\top}$ and $Y=(Y_1 \cdots Y_n)^{\top}$. The direction estimator is $\hat\beta/\norm{\hat\beta}$, while the norm estimator is $\norm{\hat\beta}$.
 
We present the estimation errors for direction and norm in \cref{thm:err-unit-beta,thm:err-norm}:
\begin{theorem}\label{thm:err-unit-beta}
 For $n > 6d$, we have
 \begin{align}
     \Mean\bracket*{\norm*{\frac{\hat\beta}{\norm{\hat\beta}} - \frac{\beta^*}{\norm{\beta^*}}}^2} \le \frac{84e^{10}\sigma^2_\xi d}{\sigma^2_X\norm{\beta^*}^2n}\paren*{1 + \frac{6}{n-6}}.
 \end{align}
\end{theorem}
\begin{theorem}\label{thm:err-norm}
    For $n > 6d$, we have
    \begin{align}
        \Mean\bracket*{\paren*{\norm{\hat\beta} - \norm{\beta^*}}^2} \le \frac{21e^{10}\sigma^2_\xi d}{\sigma^2_Xn}\paren*{1 + \frac{6}{n-6}}.
    \end{align}
\end{theorem}
The direction's estimation error~(\cref{thm:err-unit-beta}) is $O(\nicefrac{\sigma^2_\xi d}{\sigma^2_X\norm{\beta^*}^2n})$, while the norm's estimation error~(\cref{thm:err-norm}) is $O(\nicefrac{B^2\sigma^2_\xi d}{\sigma^2_Xn})$. Integrating \cref{thm:decompose,thm:err-unit-beta,thm:err-norm} yields the $\nicefrac{\sigma^2_\xi B^2dM}{n}$ term in the upper bound in \cref{thm:main-err}. The remaining part, $\nicefrac{U\sigma^2_\xi}{n}$, arises from the estimation error of $\hat\beta'_s$~(the third term in \cref{eq:decompose}), dominating other terms in \cref{eq:decompose}.

\section{Lower Bound Analyses}\label{sec:lower}
In this section, we provide a proof sketch for the lower bound, outlined in \cref{thm:main-err}. To facilitate a clear and concise presentation of the proof sketch, we introduce several notations. Let $\theta$ denote the tuple of distribution parameters $(\beta_\cdot, \mu_\cdot)$, and let $\Theta$ represent the set of all such parameters, defined as $\Theta = \dom{B}\times\dom{M}$. We use $\p_\theta$ and $\Mean_\theta$ to denote the probability and expectation operators, respectively, given $X \sim N(\mu_S, \sigma^2_X I)$ and $Y = (\beta_S, X) + \xi$, where $\xi \sim N(0, \sigma^2_\xi)$. We adopt the shorthand $\mathcal{E}(f; \theta) = \mathcal{E}(f; \beta_\cdot, \mu_\cdot)$ for $\theta = \paren{ \beta_\cdot, \mu_\cdot }$. Moreover, we define $f_\theta = \argmin_{f \in \dom{F}_{\mathrm{DP}}} R(f; \beta_\cdot, \mu_\cdot)$ for $\theta = ( \beta_\cdot, \mu_\cdot )$. For two probability distributions $\pi$ and $\pi'$, the Kullback-Leibler~(KL) divergence is denoted as $\KL(\pi, \pi') = \int \ln (\frac{d\pi}{d\pi'}(z)) \pi(dz)$. Finally, we denote the set of all $L^2$ integrable functions $f : \RealSet^d \times [M] \to \RealSet$ as $\dom{L}^2$.

By utilizing Fano's inequality, we establish a lower bound for the minimax error as presented in \cref{thm:main-err}. Due to the invariance of the distribution of $S_1,...,S_n$ under parameter alterations $\theta$, Fano's inequality can be applied after conditioning on $S_1,...,S_n$, or equivalently, $n_\cdot$. Consequently, we derive the following theorem:
\begin{theorem}\label{thm:applied-fano}
 Let $\hat\Theta \subseteq \Theta$ be a finite set of the parameters such that there exists $\epsilon > 0$ such that for any $\theta,\theta' \in \hat\Theta$, $\inf_f\mathcal{E}(f;\theta)\lor\mathcal{E}(f;\theta') \ge \epsilon$, where $\hat\Theta$ and $\epsilon$ is possibly dependent on $n_\cdot$. Let $\abs{\hat\Theta} = K$. Then, for arbitrary $\alpha > 0$ and $\delta \in (0,1)$, we have
 \begin{align}
     \mathcal{E}_n\paren*{\alpha,\delta} \ge \Mean\bracket*{\epsilon\paren*{1 - \frac{\inf_\pi\frac{1}{K}\sum_{\theta \in \hat\Theta}\KL\paren*{\pi_{\theta|n_\cdot},\pi} + \ln(2)}{\ln\paren{K}}}},
 \end{align}
 where $\pi_{\theta|n_\cdot}$ denotes the distribution of $D_n$ conditioned on $n_\cdot$ with parameter $\theta$, and the expectation is taken over $n_\cdot$.
\end{theorem}
As demonstrated in \cref{thm:applied-fano}, the lower bound for the minimax error can be obtained by constructing $\hat\Theta$ such that: 1) $\inf_f\mathcal{E}(f;\theta)\lor\mathcal{E}(f;\theta') \ge \epsilon$ for any $\theta,\theta' \in \hat\Theta$, and 2) $\inf_\pi\frac{1}{K}\sum_{\theta \in \hat\Theta}\KL\paren*{\pi_{\theta|n_\cdot},\pi} \le \ln(K/4)/2$. With the construction of such a $\hat\Theta$, a lower bound of $\Mean\bracket{\frac{\epsilon}{2}}$ is attained.

We present a theorem that establishes a tight lower bound on $\inf_f\mathcal{E}(f;\theta)\lor\mathcal{E}(f;\theta')$.
\begin{theorem}\label{thm:two-point-lower}
    Let $\theta$ and $\theta'$ be the parameters of the distributions such that $ \frac{1}{2\sigma^2_X}\norm*{\mu_s - \mu'_s}^2 \coloneqq d_s < 1$ for all $s \in [M]$. Then, we have
    \begin{align}
        \inf_{f \in \dom{L}^2}\mathcal{E}(f;\theta)\lor\mathcal{E}(f;\theta') \ge \sum_{s \in [M]}p_s\frac{\sigma^2_Xe^{-d_s}}{4}\norm*{\frac{\overline{\norm{\beta_\cdot}}\beta_s}{\norm{\beta_s}} - \frac{\overline{\norm{\beta'_\cdot}}\beta'_s}{\norm{\beta'_s}}}^2\paren*{1+\frac{d_s}{2}}^{1+\frac{d}{2}}.
    \end{align}
\end{theorem}
The term $\norm{\nicefrac{\overline{\norm{\beta_\cdot}}\beta_s}{\norm{\beta_s}} - \nicefrac{\overline{\norm{\beta'_\cdot}}\beta'_s}{\norm{\beta'_s}}}^2$ characterizes the lower bound, which is different from the characteristic term in standard linear regression, $\norm{\beta_s-\beta'_s}^2$.

We next present the construction of $\hat\Theta$. We construct $\hat\Theta$ such that each of its elements corresponds to an index from the set $\dom{V}=\cbrace{-1,1}^{M\times(d-1)}$, denoted by $\theta_v = \cbrace{\beta_{v,\cdot},\mu_{v,\cdot}} \in \hat\Theta$, where $\beta_{v,s}$ is controlled such that its norm is equivalent to a specified value $B_s$, i.e., $\norm{\beta_{v,s}} = B_s$. This construction ensures that $\hat\Theta \subset \dom{B}\times\dom{M}$. Given positive values $\epsilon_1,...,\epsilon_M$ and $B_1,...,B_M$, we construct $\hat\Theta$ as follows:
\begin{gather}
\mu_{v,s}=0, \norm{\beta_{v,s}}=B_s, \frac{\beta_{v,s,1}}{\norm{\beta_{v,s}}} = \sqrt{1 - \epsilon^2_{s}}, \textand \frac{\beta_{v,s,i}}{\norm{\beta_{v,s}}} = v_{s,i-1}\frac{\epsilon_{s}}{\sqrt{d-1}} \for i = 2,...,d. \label{eq:construct}
\end{gather}
We demonstrate the following properties for $\hat\Theta$ defined in \cref{eq:construct}.
\begin{theorem}\label{thm:distinct-params1}
Given $\epsilon_1,...,\epsilon_M > 0$ and $B_1,...,B_M > 0$, let $\hat\Theta \subset \Theta$ represent the set of parameters defined in \cref{eq:construct}. Let $\pi_{\theta|n_\cdot}$ be the distribution of the sample $D_n$ conditioned on $n_\cdot$ with the distribution parameter $\theta$. Then, we have 1) for any $v,v' \in \dom{V}$,
\begin{align}
\inf_{f \in \dom{L}^2}\mathcal{E}(f;\theta_v)\lor\mathcal{E}(f;\theta_{v'}) \ge \sum_{s\in[M]}p_s\paren*{\sum_{s' \in [M]}p_{s'}B_{s'}}^2\frac{\sigma^2_X\epsilon_s^2}{d-1}d_H(v_s,v's),
\end{align}
and 2) for $v,v' \in \dom{V}$,
\begin{align}
\KL\paren*{\pi_{\theta_v|n_\cdot},\pi_{\theta_{v'}|n_\cdot}} = \sum_{s\in[M]}\frac{2\sigma^2_XB_s^2n_s\epsilon_s^2}{\sigma^2_\xi(d-1)}d_H(v_s,v'_s).
\end{align}
\end{theorem}
By integrating \cref{thm:applied-fano,thm:two-point-lower,thm:distinct-params1} and employing the renowned Varshamov-Gilbert bound, we derive the lower bound in \cref{thm:main-err}.

\section{Conclusion}\label{sec:conclude}
This paper investigates a regression problem with $(\alpha,\delta)$-fairness consistency as a fairness constraint. Specifically, we demonstrate that, under the constraint of $(\alpha,\delta)$-fairness, the minimax optimal error scales as $\nicefrac{\sigma^2_\xi B^2dM}{n}$ up to a constant factor, when $\alpha \in (0,\nicefrac{1}{2}]$. Additionally, we provide the fair regressor that achieves this optimal error.

\parahead{Potential negative societal impacts}
Our study aims to mitigate the negative impact of regression models on social groups, rather than to cause harm. However, our results are only valid for linear models, as defined in \cref{eq:model}. Misapplication of our findings to other models may result in discriminatory treatment, which should be avoided.

\section*{Acknowledgement}
This work was partly supported by JSPS KAKENHI Grant Numbers JP23K13011 and JP23H00483.

\printbibliography

\appendix
\section{Comparison of Existing and Our Unfairness Scores}
This section compares our unfairness score with existing ones. Recall that our unfairness score is defined as the maximum Wasserstein distance between any two distributions $\nu_{f|s}$ and $\nu_{f|s'}$ over all pairs of groups $s$ and $s'$, as follows:
\begin{align}
    U(f) = \max_{s,s'\in[M]}W_2(\nu_{f_n|s}, \nu_{f|s'}).
\end{align}
In contrast, \textcite{Agarwal2019FairAlgorithms,Chzhen2020FairBarycenters,Chzhen2020FairRecalibration} use the Kolmogorov distance $D_{\mathrm{Kol}}$ to measure unfairness, which is defined as:
\begin{align}
    U_{\mathrm{Kol}}(f) = \max_{s,s'\in[M]}D_{\mathrm{Kol}}(\nu_{f|s}, \nu_{f|s'}).
\end{align}
The difference between our score and $U_{\mathrm{Kol}}(f)$ is solely the choice of distance metric. Our score utilizes the Wasserstein distance, while $U_{\mathrm{Kol}}(f)$ uses the Kolmogorov distance. This difference arises mainly from technical reasons.

In addition, \textcite{Chzhen2022ARegression} proposed another unfairness score, denoted by $U_{\mathrm{Avg}W_2}(f)$, which is defined as the average of the Wasserstein distance, as follows:
\begin{align}
U_{\mathrm{Avg}W_2}(f) = \inf_{\nu}\sum_{s\in[M]}p_sW_2\paren*{\nu_{f|s},\nu}.
\end{align}
Here, the score places more emphasis on the major groups, as reflected by the weight of $p_s$. This may not be desirable if the unfairness is more prevalent in the minority groups, which may be common in real-world scenarios.

\section{Estimator Details}
\begin{figure}[t]
    \centering
    \includegraphics[width=.7\textwidth]{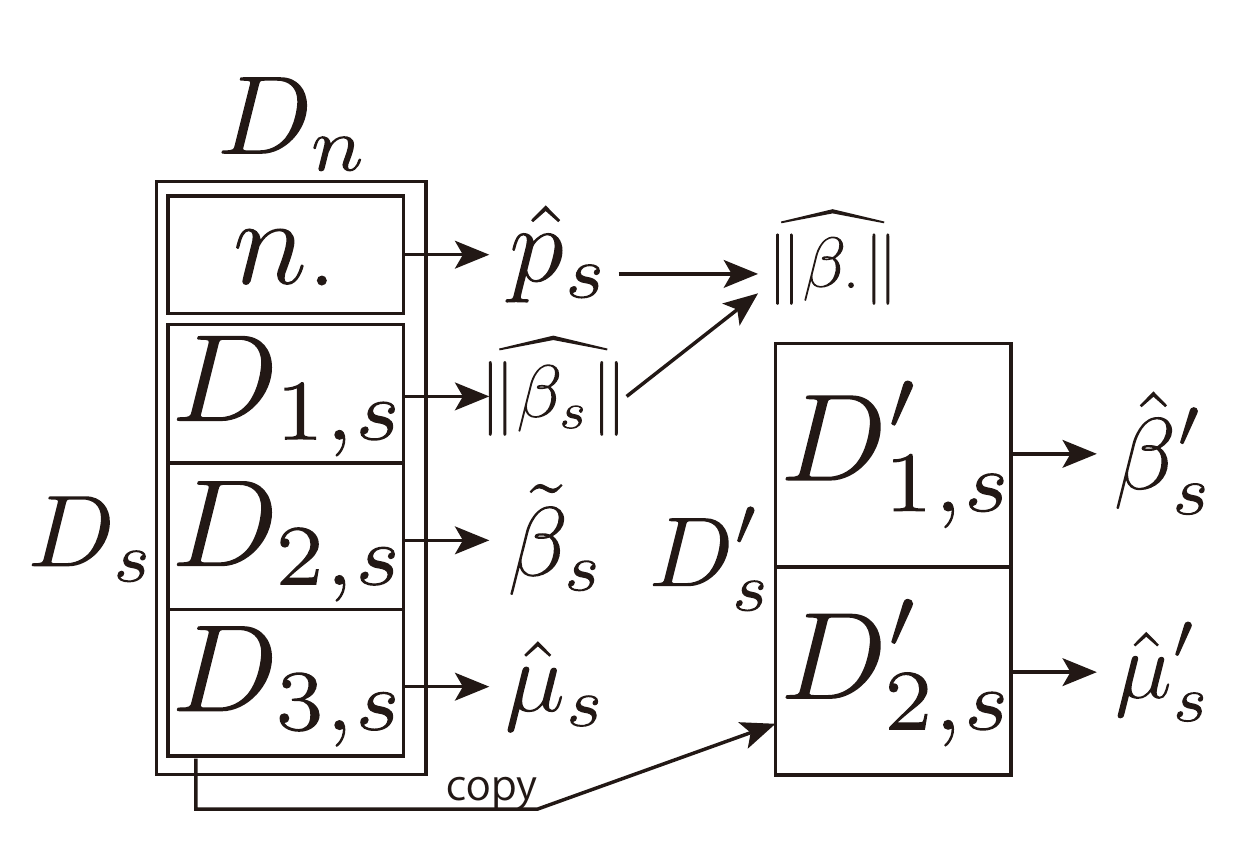}
    \caption{Sample splitting for constructing estimators. }
    \label{fig:estimator-const}
\end{figure}
\begin{algorithm}[t]
    \SetKwInOut{Input}{Input}
    \SetKwInOut{Output}{Output}
    \Input{The sample $D_n=\cbrace{(X_i,S_i,Y_i}_{i=1}^n$.}
    \Output{The regressor $\hat{f}_n$.}
    \For{$s \gets 1$ \KwTo $M$}{
        Calculate $n_s$ and construct the group-wise sample $D_s$ and its duplicate $D'_s$ \;
        Partition $D_s$ into equal-sized subsets: $D_{1,s}$, $D_{2,s}$, and $D_{3,s}$ \;
        Compute the estimands $\hat{p}_s$ from $n_\cdot$, $\widehat{\norm{\beta_s}}$ from $D_{1,s}$, $\tilde\beta_s$ from $D_{2,s}$, and $\hat\mu_s$ from $D_{3,s}$ \;
        Partition $D'_s$ equally into $D'_{1,s}$ and $D'_{2,s}$ \;
        Compute the estimands $\hat\beta'_s$ from $D'_{1,s}$ and $\hat\mu'_s$ from $D'_{2,s}$ \;
    }
    Calculate $\hat{f}_n$ using \cref{eq:loose-reg} \;
    \Return{$\hat{f}_n$}
    \caption{Algorithm of the proposed optimal estimator.}\label{alg:est-alg}
\end{algorithm}

This section describes the construction of our optimal estimator in detail. Recall that our estimator is a plugin estimator in which we first estimate the parts of the terms in \cref{eq:opt-cls} and then substitute them into \cref{eq:opt-cls}. Specifically, we construct estimators for$\widehat{\norm{\beta_\cdot}}$, $\tilde\beta_s$, $\hat{p}_{s'}$, $\hat\beta'_{s'}$, and $\hat\mu'_{s'}$, where they correspond to the terms in \cref{eq:opt-cls} as follows:
\begin{align}
    \underbracket{\overline{\norm{\beta^*_\cdot}}}_{\widehat{\norm{\beta_\cdot}}}\abrace*{\underbracket{\frac{\beta^*_s}{\norm{\beta^*_s}}}_{\tilde\beta_s},x - \underbracket{\mu_s}_{\hat\mu_s}} + \sum_{s'\in[M]}\underbracket{p_{s'}}_{\hat{p}_{s'}}\abrace*{\underbracket{\beta^*_{s'}}_{\hat\beta'_{s'}},\underbracket{\mu_{s'}}_{\hat\mu'_{s'}}}.
\end{align}

For analysis purposes, we split the sample into several subsets and pass each subset to the corresponding estimator~(the correspondence is explained later). \cref{fig:estimator-const} shows an overview of the sample splitting and the correspondence between the subsets and estimators. We construct the histogram of the sensitive attribute $S_i$ from the sample $D_n$, denoted as $n_\cdot = (n_1,...,n_M)$, where $n_s = \abs{\cbrace{i \in [n] : S_i = s}}$~(upper left in \cref{fig:estimator-const}). We also construct group-wise samples $D_s = \cbrace{(X_i,Y_i) : i \in [n], S_i = s}$. For each $s \in [M]$, we divide $D_s$ into $D_{1,s}$, $D_{2,s}$, and $D_{3,s}$ such that $\abs{D_{b,s}} \coloneqq n_{b,s} \ge \floor{n_s/3}$ for $b \in [3]$~(lower left in \cref{fig:estimator-const}). We use $n_\cdot$, $D_{1,s}$, $D_{2,s}$, and $D_{3,s}$ to estimate $\hat{p}s$, $\widehat{\norm{\beta_s}}$, $\tilde\beta_s$, and $\hat\mu_s$, respectively, where $\widehat{\norm{\beta_s}}$ is the estimator for $\norm{\beta^*s}$. We obtain $\widehat{\norm{\beta_\cdot}}$ from the combination of $\hat{p}_s$ and $\widehat{\norm{\beta_s}}$~(middle in \cref{fig:estimator-const}). Furthermore, for each $s \in [M]$, we create a copy of $D_s$, denoted as $D's$, and divide it into $D'{1,s}$ and $D'{2,s}$ such that $\abs{D'{b,s}} \coloneqq n'{b,s} \ge \floor{n_s/2}$ for $b \in [2]$. We use $D'{1,s}$ and $D'_{2,s}$ to estimate $\hat\beta'_s$ and $\hat\mu'_s$, respectively~(right in \cref{fig:estimator-const}).

We describe the construction of each estimator below. We first define some notations. Let the $i$th element of $D_{b,s}$ and $D'_{b,s}$ be denoted as $(X_{b,s,i},Y_{b,s,i})$ and $(X'_{b,s,i},Y'_{b,s,i})$, respectively. We use the matrix notations $X_{b,s} = (X_{b,s,1} \cdots X_{b,s,n_{b,s}})^\top$, $X'_{b,s} = (X'_{b,s,1} \cdots X'_{b,s,n'_{b,s}})^\top$, $Y_{b,s} = (Y_{b,s,1} \cdots Y_{b,s,n_{b,s}})^\top$, and $Y'_{b,s} = (Y'_{b,s,1} \cdots Y'_{b,s,n'_{b,s}})^\top$. We define the ordinary least square estimators for the subset of the sample $D_{b,s}$ and $D'_{b,s}$ as follows:
\begin{align}
    \hat\beta_{b,s} =& \paren*{\frac{1}{n_{b,s}}X_{b,s}^{\top}X_{b,s}}^{-1}\paren*{\frac{1}{n_{b,s}}X_{b,s}^{\top}Y_{b,s}}, \\
    \hat\beta'_{b,s} =& \paren*{\frac{1}{n'_{b,s}}(X'_{b,s})^{\top}X'_{b,s}}^{-1}\paren*{\frac{1}{n'_{b,s}}(X'_{b,s})^{\top}Y'_{b,s}}.
\end{align}
We construct each estimator as follows:
\begin{description}[leftmargin=!, labelwidth=!]
    \item[($\hat{p}s$)] We use the empirical mean defined as $\hat{p}_s = \nicefrac{n_s}{n}$.
    \item[($\widehat{\norm{\beta_s}}$)] We use the norm of the OLS estimator. We define $\widehat{\norm{\beta_s}} = \norm{\hat\beta_{1,s}}$ if $n_s > 18d$, and $\widehat{\norm{\beta_s}}=0$ otherwise.
    \item[($\widehat{\norm{\beta_\cdot}}$)] Since $\overline{\norm{\beta^*_\cdot}} = \sum_{s \in [M]}p_s\norm{\beta^*_s}$, we construct its estimator as $\widehat{\norm{\beta_\cdot}} = \sum_{s \in [M]}\hat{p}_s\widehat{\norm{\beta_s}}$.
    \item[($\tilde\beta_s$)] We use the normalized ordinary least square estimator; $\tilde\beta_s = \hat\beta_{2,s}/\norm{\hat\beta_{2,s}}$ if $n_s > 18d$, and $\tilde\beta_s = 0$ otherwise.
    \item[($\hat\mu_s$)] We use the empirical mean; $\hat\mu_s = \frac{1}{n_{3,s}}\sum_{i=1}^{n_{3,s}}X_{3,s,i}$.
    \item[($\hat\beta'_s$)] We employ the ordinary least square estimator; $\hat\beta'_s = \hat\beta'_{1,s}$ if $n_{s} > 12d$, and $\hat\beta'_s = 0$ otherwise.
    \item[($\hat\mu'_s$)] We use the empirical mean; $\hat\mu'_s = \frac{1}{n'_{2,s}}\sum_{i=1}^{n'_{2,s}}X'_{2,s,i}$ if $n_s > 12d$, and $\hat\mu'_s = 0$ otherwise.
\end{description}
Some estimators change their behavior based on the condition $n_s > 18d$ or $n_s > 12d$, which is done for technical purposes in later analyses.

Recall that the final regressor is constructed as follows:
\begin{align}
\hat{f}_n(x,s) = \widehat{\norm{\beta_{\cdot}}}\abrace*{\tilde\beta_s, x - \hat\mu_s} + \sum_{s'\in[M]}\hat{p}_{s'}\abrace*{\hat\beta'_s,\hat\mu'_s}.
\end{align}
\cref{alg:est-alg} shows the algorithm for our estimator.

\section{Bayes Optimal Regressor under Our Modell}
This section presents the proof of \cref{lem:bayes-opt-linear}, demonstrating the Bayes optimal regressor under the model \cref{eq:model}. To establish this, we make use of a key result from the work of \textcite{Chzhen2020FairBarycenters}:
\begin{theorem}[\textcite{Chzhen2020FairBarycenters}]\label{thm:chzhen-barycenters}
 Assume, for each $s \in [M]$, $\nu_{f^*|s}$ has a density. Then, 
 \begin{align}
     \inf_{f: \mathrm{DP}}\Mean\bracket*{\paren*{f(X,S) - f^*(X,S)}^2} = \inf_{\nu}\sum_{s \in [M]}p_sW^2_2(\nu_{f^*|s},\nu). 
 \end{align}
 where the infimum is taken over all the regressors that satisfy the demographic parity. Moreover, letting $f^*_{\mathrm{DP}}$ and $\nu^*$ be the minimizer of the lhs and rhs, respectively, we have $\nu_{f^*_{\mathrm{DP}}} = \nu^*$ and
 \begin{align}
     f^*_{\mathrm{DP}}(x,s) = \paren*{\sum_{s'\in[M]}p_{s'}F^{-1}_{f^*|s'}}\circ F_{f^*|s}\paren*{f^*(x,s)}. \label{eq:opt-g}\sbox0{\ref{eq:opt-g}}
 \end{align}
\end{theorem}
Here, we denote $\nu_{f^*}$ as the distribution of $f^*(X,S)$, $F_{f|s}$ as the cumulative distribution function of $f(X,S)$ conditioned on $S=s$, and $F^{-1}_{f|s}$ as the inverse cumulative distribution function, given by $F^{-1}_{f|s}(t) = \inf\cbrace{y \in \RealSet | F_{f|s}(y) \ge t}$.

Building upon the results of \cref{thm:chzhen-barycenters}, we establish the proof of \cref{lem:bayes-opt-linear}.
\begin{proof}[Proof of \cref{lem:bayes-opt-linear}]
 Building upon \cref{thm:chzhen-barycenters}, we can derive the Bayes optimal regressor under the model \cref{eq:model} by obtaining closed expressions of the cumulative and inverse cumulative distribution functions $F_{f^*|s}$ and $F^{-1}_{f^*|s}$. To obtain these closed forms, we apply certain transformations to $f^*(X,S)$ that render it a random variable following a standard normal distribution. Let $\Phi$ and $\Phi^{-1}$ be the CDF and inverse CDF of the standard normal distribution, respectively. Through elementary calculations, we have:
 \begin{align}
     F_{f^*|s}(t) =& \p\cbrace*{f^*(X,S) \le t | S=s} \\
     =& \p\cbrace*{\abrace*{\beta^*_s, X} \le t | S=s} \\
     =& \p\cbrace*{\frac{1}{\sigma_X\norm{\beta^*_s}}\abrace*{\beta^*_s, X - \mu_s} \le \frac{1}{\sigma_X\norm{\beta^*_s}}\paren*{t - \abrace*{\beta^*_s, \mu_s}} | S=s}.
 \end{align}
 Here, we can readily observe that $\frac{1}{\sigma_X\norm{\beta^*_s}}\abrace*{\beta^*_s, X - \mu_s}$ follows the standard normal distribution under conditioned on $S=s$, as $X \sim N(\mu_s,\sigma_XI)$ conditioned on $S=s$. Consequently, we have 
 \begin{align}
     F_{f^*|s}(t) = \Phi\paren*{\frac{1}{\sigma_X\norm{\beta^*_s}}\paren*{t - \abrace*{\beta^*_s, \mu_s}}}. \label{eq:linear-cdf}
 \end{align}
 The inverse function of $F^{-1}_{f^*|s}(t)$ can be obtained by equating the right-hand side to $p$ and solving the resulting equation for $t$, which leads to
 \begin{align}
     F^{-1}_{f^*|s}(p) = \sigma_X\norm{\beta^*_s}\Phi^{-1}(p) + \abrace*{\beta^*_s,\mu_s}. \label{eq:linear-icdf}
 \end{align}
 By substituting \cref{eq:linear-cdf,eq:linear-icdf} into \cref{eq:opt-g} in \cref{thm:chzhen-barycenters}, we obtain the desired claim.
\end{proof}

\section{Details of Fairness Analysis}
In this section, we provide evidence of the guarantee of our estimator's fairness consistency. Specifically, we present the following theorem.
\begin{theorem}\label{thm:loose-fair} 
 For any $\delta \in (0,1]$, the regressor in \cref{eq:loose-reg} is $(\nicefrac{1}{2},\delta)$-consistently fair.
\end{theorem}
We prove the above claim by utilizing \cref{thm:consist-bound}, which is shown in the main body, as follows:
\begin{proof}[Proof of \cref{thm:loose-fair}]
 We can confirm the claim by comparing the bound obtained in \cref{thm:consist-bound} with the definition of $(\frac{1}{2},\delta)$-consistent fairness in \cref{def:consistent}. In particular, we can set $C = 4B\sigma_X\sqrt{\frac{48\ln(M/\delta)}{\min_{s\in[M]}p_s}}$ and $n_0 = n \ge 48\ln(M/\delta)/\min_sp_s$ to satisfy the definition of $(\frac{1}{2},\delta)$-consistent fairness.
\end{proof}

Next, we provide the proof of \cref{thm:consist-bound}. To this end, we prove the following two theorems:
\begin{theorem}\label{thm:wd-g}
  Let $\hat{f}_{n}$ be the estimator of $f^*_{\mathrm{DP}}$ defined in \cref{eq:loose-reg}. Then, almost surely, we have
  \begin{align}
      W_2\paren*{\nu_{\hat{f}_{n}|s},\nu_{\hat{f}_{fn}|s'}} \le 2B\paren*{\abs*{\abrace*{\tilde\beta_s,\mu_{s} - \hat\mu_{s}}}\lor\abs*{\abrace*{\tilde\beta_{s'},\mu_{s'} - \hat\mu_{s'}}}}.
  \end{align}
\end{theorem}
\begin{theorem}\label{thm:proj-mean}
 If $n \ge (48\ln(M/\delta) - 36d)/\min_sp_s$, we have for $\delta \in (0,1)$,
 \begin{align}
     \p\cbrace*{\exists s \in [M], \abs*{\abrace*{\tilde\beta_s,\mu_{s} - \hat\mu_{s}}} > \sigma_X\sqrt{\frac{48\ln(M/\delta)}{\min_snp_s + 36d}}} \le \delta.
 \end{align}
\end{theorem}
Combining \cref{thm:wd-g,thm:proj-mean} immediately yields \cref{thm:consist-bound}.
\begin{proof}[Proof of \cref{thm:wd-g}]
 This proof investigates the distribution of $\nu_{\hat{f}{n}|s}$. It is straightforward to verify that, conditioned on $S=s$, $\hat{f}_n(X,S)$ follows the Gaussian distribution with mean
 \begin{align}
     \widehat{\norm{\beta_\cdot}}\abrace*{\tilde\beta_s,\mu_s - \hat\mu_s} + \sum_{s'\in[M]}\hat{p}_{s'}\abrace*{\hat\beta'_{s'},\hat\mu'_{s'}},
 \end{align}
 and variance
 \begin{align}
     \sigma_X^2\widehat{\norm{\beta_\cdot}}^2.
 \end{align}
 We can thus evaluate the Wasserstein distance between the distributions $\nu_{\hat{f}{n}|s}$ and $\nu{\hat{f}{n}|s'}$ using the Wasserstein distance between Gaussian distributions. Given two Gaussian distributions $N(\mu,\sigma^2)$ and $N(\mu',\sigma'^{2})$, the 2-Wasserstein distance between them are obtained~\autocite{Olkin1982TheMatrices} as
 \begin{align}
     W_2^2(N(\mu,\sigma^2),N(\mu',\sigma^2)) = \paren*{\mu - \mu'}^2 + \paren*{\sigma - \sigma'}^2.
 \end{align}
Therefore, we have
 \begin{align}
    W^2_2(\nu_{\hat{f}_{n}|s},\nu_{\hat{f}_{n}|s'}) =&  \widehat{\norm{\beta_\cdot}}^2\paren*{\abrace*{\tilde\beta_s,\mu_s - \hat\mu_s} - \abrace*{\tilde\beta_{s'},\mu_{s'} - \hat\mu_{s'}}}^2 \\
    \le& 4B^2\paren*{\abs*{\abrace*{\tilde\beta_s,\mu - \hat\mu}}\lor\abs*{\abrace*{\tilde\beta_{s'},\mu - \hat\mu}}}^2,
 \end{align}
 which concludes the claim.
\end{proof}
\begin{proof}[Proof of \cref{thm:proj-mean}]
We start by deriving the concentration inequality for $\abrace{\tilde\beta_s,\mu_s - \hat\mu_s}$ conditioned on $\tilde\beta_s$ and $n_\cdot$. Note that $\tilde\beta_s = 0$ if $n_s \le 18d$. Conditioning on $\tilde\beta_s$ and $n_\cdot$, we observe that $\abrace{\tilde\beta_s,\mu_s - \hat\mu_s}$ follows a Gaussian distribution with mean zero and variance $\sigma^2_X/n_{3,s}$. Therefore, for any $s \in [M]$ and $t > 0$,
 \begin{align}
     \p\cbrace*{\abrace*{\tilde\beta_s,\mu_s - \hat\mu_s} > t \middle| \tilde\beta_s,n_\cdot} \le \ind\cbrace*{n_s > 18d}\exp\paren*{-\frac{n_{3,s}t^2}{2\sigma^2_X}}.
 \end{align}
 Taking the expectation with respect to $\tilde\beta_s$ and using the fact that $n_{3,s} \ge \floor{n_s/3} \ge n_s/6$ for $n_s \ge 6$, we obtain the following inequality for $s \in [M]$ and $t > 0$:
 \begin{align}
     \p\cbrace*{\abrace*{\tilde\beta_s,\mu_s - \hat\mu_s} > t \middle| n_\cdot} \le \ind\cbrace*{n_s > 18d}\exp\paren*{-\frac{n_st^2}{12\sigma^2_X}}.
 \end{align}
 Using the union bound for $t > 0$, we have
 \begin{align}
     \p\cbrace*{\exists s \in [M], \abrace*{\tilde\beta_s,\mu_s - \hat\mu_s} > t \middle| n_\cdot} \le \sum_{s \in [M]}\ind\cbrace*{n_s > 18d}\exp\paren*{-\frac{n_st^2}{12\sigma^2_X}}. \label{eq:proj-mean-bound1}
 \end{align}
 
 We now derive a sufficient condition on $t$ such that the expectation of the right-hand side in \cref{eq:proj-mean-bound1} is less than $\delta$. First, we note that
 \begin{align}
     & \ind\cbrace*{n_s > 18d}\exp\paren*{-\frac{n_st^2}{12\sigma^2_X}} \\
     \le& \ind\cbrace*{n_s > 18d}\exp\paren*{-\frac{(n_s+18d)t^2}{12\sigma^2_X}\frac{n_s}{n_s+18d}} \\
     \le& \exp\paren*{-\frac{(n_s+18d)t^2}{24\sigma^2_X}}
 \end{align}
 Taking the expectation and substituting the moment-generating function of the binomial distribution, we obtain
 \begin{align}
     & \Mean\bracket*{\exp\paren*{-\frac{(n_s+18d)t^2}{24\sigma^2_X}}} \\
     \le& \exp\paren*{-\frac{18dt^2}{24\sigma^2_X}}\paren*{1-p_s + p_se^{-\frac{t^2}{24\sigma^2_X}}}^n \\
     \le&\exp\paren*{-\frac{18dt^2}{24\sigma^2_X}-np_s\paren*{1-e^{-\frac{t^2}{24\sigma^2_X}}}}.
 \end{align}
 Since $1-e^{-x} \ge (1-e^{-1})x$ for $x \in [0,1]$, if $t^2/24\sigma^2_X \le 1$, we have
 \begin{align}
     & \Mean\bracket*{\exp\paren*{-\frac{(n_s+18d)t^2}{24\sigma^2_X}}} \\
     \le&\exp\paren*{-\frac{18dt^2}{24\sigma^2_X}-\frac{(1-e^{-1})np_st^2}{24\sigma^2_X}}.
 \end{align}
 Hence, $\Mean\bracket{\exp\paren{-\frac{(n_s+18d)t^2}{24\sigma^2_X}}} \le \nicefrac{\delta}{M}$ if $t \ge \sigma_X\sqrt{\frac{24\ln(M/\delta)}{(1-e^{-1})np_s + 18d}} \le \sigma_X\sqrt{\frac{48\ln(M/\delta)}{\min_snp_s + 36d}}$ because $(1-e^{-1}) \ge 1/2$. To ensure $t^2/24\sigma^2_X \le 1$, we require $n \ge (48\ln(M/\delta) - 36d)/\min_sp_s$.
\end{proof}

\section{Proofs for Norm and Direction Estimators}
This section presents the proofs for \cref{thm:err-unit-beta} and \cref{thm:err-norm}. Our strategy for proving these theorems is to use the hyperellipsoid to interpret the distribution of the OLS estimator. Specifically, we begin by defining $\Sigma_n = \frac{1}{n}X^{\top}X$ and expressing the OLS estimator $\hat{\beta}$ as
\begin{align}
    \hat\beta = \beta^* + \Sigma_n^{-1}\paren*{\frac{1}{n}X^{\top}\xi}, \label{eq:ols-est}
\end{align}
where $\xi$ follows a zero-mean Gaussian distribution. \cref{eq:ols-est} shows that, conditioned on $X$, $\hat\beta$ follows a multivariate Gaussian distribution with mean $\beta^*$ and covariance matrix $\frac{\sigma^2_\xi}{n}\Sigma_n^{-1}$. We establish that, under the condition $\norm{\hat\beta} = r$, $\hat\beta$ is supported on a hyperellipsoid $E(r,\beta^, \frac{n}{\sigma^2_\xi}\Sigma_n)$, where $E(r,c,A) = \cbrace{x \in \RealSet^d : \paren{x - c}^{\top}A\paren{x-c} \le r}$ denotes the hyperellipsoid with $r > 0$, $c \in \RealSet^d$, and a symmetric and positive-definite matrix $A \in \RealSet^{d\times d}$.

To prove \cref{thm:err-unit-beta} and \cref{thm:err-norm}, we adopt the following strategy. First, we provide an approximation of the hyperellipsoid $E(r,c,A)$ using the maximum eigenvalue of $A^{-1}$, i.e., $\lambda_{\max}(A^{-1})$. In our context, $A=\frac{n}{\sigma^2_\xi}\Sigma_n$, and we then focus on the concentration inequalities regarding $\lambda_{\max}(\Sigma_n^{-1})$. Finally, we combine these tools to prove both theorems.

\parahead{Lemmas regarding hyperellipsoid}
We present two lemmas that relate to the approximation of the hyperellipsoid $E(r,c,A)$. Specifically, we demonstrate the following two lemmas:
\begin{lemma}\label{lem:ell-sph}
  For $r > 0$, $c \in \RealSet^d$, and a symmetric and positive-definite matrix $A \in \RealSet^{d\times d}$, we have $E(r,c,A) \subseteq E(r\lambda_{\max}(A^{-1}),c,I)$.
\end{lemma}
\begin{lemma}\label{lem:ell-angle}
 For $r > 0$, $c \in \RealSet^d$, and a symmetric and positive-definite matrix $A \in \RealSet^{d\times d}$, if $r\lambda_{\max}(A^{-1}) \le \norm{c}^2$, we have
 \begin{align}
     \inf_{x \in E(r,c,A)}\abrace*{\frac{c}{\norm{c}}, \frac{x}{\norm{x}}} \ge \sqrt{1 - \frac{r}{\norm{c}^2}\lambda_{\max}(A^{-1})}.
 \end{align}
\end{lemma}
These lemmas provide insight into the approximation of the hyperellipsoid $E(r,c,A)$ for a given positive value of $r$, vector $c$ in $\RealSet^d$, and positive-definite symmetric matrix $A$ in $\RealSet^{d\times d}$. \cref{lem:ell-sph} states that the hyperellipsoid $E(r,c,A)$ is contained within a hyperellipsoid $E(r\lambda_{\max}(A^{-1}),c,I)$. \cref{lem:ell-angle} shows that, under certain conditions, the minimum angle between a point in $E(r,c,A)$ and the vector $c$ is bounded below by a quantity that depends on $r$, $c$, and $A$.

\begin{proof}[Proof of \cref{lem:ell-sph}]
It is trivial that $A - \lambda_{\min}(A)I$ is positive semi-definite. Equivalently, we have for any $x \in \RealSet^d$,
 \begin{align}
     & x^\top\paren*{A - \lambda_{\min}(A)I}x \ge 0 \\
     \iff& x^\top A x \ge x^\top \lambda_{\min}(A)I x. \label{eq:ell-sph-1}
 \end{align}
 From \cref{eq:ell-sph-1}, for any $x \in E(r,c,A)$, we have
 \begin{align}
     x^\top \lambda_{\min}(A)I x \le x^\top A x \le r.
 \end{align}
 Hence, for any $x \in E(r,c,A)$, we have
 \begin{align}
     x^\top I x \le \frac{r}{\lambda_{\min}(A)} = \lambda_{\max}(A^{-1})r,
 \end{align}
 which indicates $x \in E(r\lambda_{\max}(A^{-1}),c,I)$.
\end{proof}
\begin{proof}[Proof of \cref{lem:ell-angle}]
Let $\bar{c} = c/\norm{c}$, and define a set $\bar{E}(r,c,A) = \cbrace{x \in \mathbb{S}{d-1}: \exists \gamma > 0, \gamma x \in E(r,c,A)}$. Then, $x \in \bar{E}(r,c,A)$ if and only if
 \begin{align}
     \inf_{\gamma > 0}\paren*{\gamma x - c}^{\top}A\paren*{\gamma x-c} \le r. \label{eq:ell-angle-cond}
 \end{align}
 We can rewrite the left-hand side of \cref{eq:ell-angle-cond} as
 \begin{align}
     & \gamma^2 \abrace*{x, A x} - 2\gamma\abrace*{c, A x} + \abrace*{c, A c} \\
     =&  \abrace*{x, A x}\paren*{\gamma - \frac{\abrace*{c, A x}}{\abrace*{x, A x}}}^2 + \abrace*{c, A c} - \frac{\abrace{x,Ac}^2}{\abrace{x,Ax}}.
 \end{align}
 Hence, 
 \begin{align}
     \inf_{\gamma > 0}\paren*{\gamma x - c}^{\top}A\paren*{\gamma x-c} = \abrace*{c, A c} - \frac{\paren{\abrace{x,Ac}\lor 0}^2}{\abrace{x,Ax}}.
 \end{align}
 Consequently, $x \in \bar{E}(r,c,A)$ if and only if
 \begin{align}
    \abrace{x,Ax}\paren*{\abrace*{\bar{c},A\bar{c}} - \frac{r}{\norm{c}^2}} \le \paren*{\abrace{x,A\bar{c}}\lor 0}^2. \label{eq:ell-angle-cond2}
 \end{align}
 
 From \cref{lem:ell-sph}, we have
 \begin{align}
     \inf_{x \in E(r,c,A)}\abrace*{\bar{c}, \frac{x}{\norm{x}}} \ge& \inf_{x \in E(\lambda_{\max}(A^{-1})r,c,I)}\abrace*{\bar{c}, \frac{x}{\norm{x}}} \\
     =& \inf_{x \in \bar{E}(\lambda_{\max}(A^{-1})r,c,I)}\abrace*{\bar{c}, x}. \label{eq:ell-angle-lower}
 \end{align}
 By \cref{eq:ell-angle-cond2}, $x \in \bar{E}(\lambda_{\max}(A^{-1})r,c,I)$ if and only if
 \begin{align}
     1 - \frac{r}{\norm{c}^2}\lambda_{\max}(A^{-1}) \le \paren*{\abrace{x,\bar{c}}\lor 0}^2. \label{eq:ell-angle-cond3}
 \end{align}
 Combining \cref{eq:ell-angle-lower,eq:ell-angle-cond3} and the assumption yields the claim.
\end{proof}

\parahead{Least eigenvalue of the empirical covariance matrix}
The previous lemmas, \cref{lem:ell-sph,lem:ell-angle}, provide valuable insight into analyzing the randomness regarding $\xi$. However, to account for the randomness of $X$, we must also control the lower bound on the least eigenvalue of $A$ in \cref{lem:ell-sph,lem:ell-angle}, which corresponds to the least eigenvalue of $\frac{1}{n}X^\top X$ in our context. To this end, we leverage the high probability bound presented by \textcite{Mourtada2022ExactMatrices} based on the small-ball condition. We state the following probabilistic bound and expectation bound.
\begin{lemma}\label{lem:min-eigvalue-sig}
    For $\mu \in \RealSet^d$ and $\sigma^2_X > 0$, let $X_1,...,X_n \iidsim N(\mu,\sigma^2_XI)$, and let $X=(X_1 \cdots X_n)^{\top}$. Then, for $n > 6d$, we have
    \begin{align}
        \p\cbrace*{\lambda_{\min}\paren*{\frac{1}{n}X^{\top}X} < t} \le \paren*{\frac{21e^{10}}{\sigma^2_X}t}^{\nicefrac{n}{6}}.
    \end{align}
\end{lemma}
\begin{lemma}\label{lem:exp-max-eig-inv}
    For $\mu \in \RealSet^d$ and $\sigma^2_X > 0$, let $X_1,...,X_n \iidsim N(\mu,\sigma^2_XI)$, and let $X=(X_1 \cdots X_n)^{\top}$. Then, for $n > 6d$, we have
    \begin{align}
        \Mean\bracket*{\lambda_{\max}\paren*{\paren*{\frac{1}{n}X^{\top}X}^{-1}}} \le \frac{21e^{10}}{\sigma^2_X}\paren*{1 + \frac{6}{n - 6}}.
    \end{align}
\end{lemma}

To prove \cref{lem:min-eigvalue-sig}, we utilize Corollary 3 in \textcite{Mourtada2022ExactMatrices}. Specifically, we use the following theorem.
\begin{theorem}[Corollary 3 in \textcite{Mourtada2022ExactMatrices}]\label{thm:eig-inq-tail}
    Let $X$ be a random vector in $\RealSet^d$ such that $\Mean\bracket{\norm{X}^2} < +\infty$, and let $\Sigma = \Mean\bracket{XX^\top}$. Let $\hat\Sigma_n = \frac{1}{n}\sum_{i=1}^n X_iX_i^\top$, where $X_i$ are i.i.d. copies of $X$. Given $C > 0$ and $\alpha \in (0,1]$, assume that for every $\theta \in \RealSet^d\setminus\cbrace{0}$ and $t > 0$, 
    \begin{align}
        \p\cbrace*{\abrace*{\theta, X}^2 \le t^2\norm*{\Sigma^{\nicefrac{1}{2}}\theta}^2} \le (Ct)^\alpha. \label{eq:small-ball}
    \end{align}
    Then, if $\nicefrac{d}{n} \le \nicefrac{\alpha}{6}$, for every $t > 0$,
    \begin{align}
        \hat\Sigma_n \succeq t\Sigma
    \end{align}
    with probability at least $1-(C't)^{\nicefrac{\alpha n}{6}}$, where $C' = 3C^4e^{1+\nicefrac{9}{\alpha}}$.
\end{theorem}
\cref{eq:small-ball} is known as the small-ball condition.

\begin{proof}[Proof of \cref{lem:min-eigvalue-sig}]
     To take an advantage of \cref{thm:eig-inq-tail}, we need to ensure that $X_i$ satisfies the small-ball condition in \cref{eq:small-ball}. Let $\Sigma_n = \frac{1}{n}X^{\top}X$. Then, the expected value of $\Sigma_n$ is equal to $\sigma^2_XI + \mu\mu^{\top} \coloneqq \Sigma$, i.e., $\Mean\bracket*{\Sigma_n} = \Sigma$. Given $\theta \in \RealSet^d\setminus\cbrace{0}$, $\abrace{\theta, X_i}^2/\sigma^2_X\norm{\theta}^2$ follows the non-central $\chi^2$ distribution with degree of freedom $1$ and non-centrality parameter $\abrace{\theta, \mu}^2/\sigma^2_X\norm{\theta}^2$. Consequently, we verify the satisfication of the small-ball condition of $X_i$ by confirming that for a random variable $Z$ following the non-central $\chi^2$ distribution with degree of freedom $1$ and non-centrality parameter $\lambda^2$, there exists $C$ and $\alpha \in (0,1]$ such that 
    \begin{align}
        \p\cbrace*{Z \le t^2} \le (Ct)^\alpha.
    \end{align}
    
    The cumulative distribution fucntion of the non-central $\chi^2$ distribution with degree of freedom 1 has a closed-form using the error function~(See \autocite{JankovMasirevic2017OnDistribution} and references therein). Specifically, letting $\mathrm{erf}(z)$ be the error function, defined as 
    \begin{align}
        \mathrm{erf}(z) = \frac{2}{\sqrt{\pi}}\int_0^ze^{-x^2}dx,
    \end{align}
    the cumulative distribution function of $Z$ is obtained as
    \begin{align}
        \p\cbrace*{Z \le t^2} = \frac{1}{2}\paren*{\mathrm{erf}\paren*{\frac{t - \lambda}{\sqrt{2}}} + \mathrm{erf}\paren*{\frac{t + \lambda}{\sqrt{2}}}}.
    \end{align}
    Since $e^{-x^2}$ is an even function, we have
    \begin{align}
        \p\cbrace*{Z \le t^2} =& \frac{1}{\sqrt{2\pi}}\paren*{\int_{0}^{\lambda+t}e^{-\nicefrac{x^2}{2}}dx + \int_{0}^{t-\lambda}e^{-\nicefrac{x^2}{2}}dx} \\
        =& \frac{1}{\sqrt{2\pi}}\paren*{\int_{0}^{\lambda+t}e^{-\nicefrac{x^2}{2}}dx + \int_{\lambda-t}^0e^{-\nicefrac{x^2}{2}}dx} \\
        =& \frac{1}{\sqrt{2\pi}}\paren*{\int_{0}^{\lambda+t}e^{-\nicefrac{x^2}{2}}dx - \int_{0}^{\lambda-t}e^{-\nicefrac{x^2}{2}}dx} \\
        =& \frac{1}{\sqrt{2\pi}}\int_{\lambda-t}^{\lambda+t}e^{-\nicefrac{x^2}{2}}dx.
    \end{align}
    Noting that $e^{-\nicefrac{x^2}{2}} \le e^{-\frac{\paren{0\lor\paren{\lambda-t}}^2}{2}}$ for $x \in (\lambda-t, \lambda+t)$, we have
    \begin{align}
        \p\cbrace*{Z \le t^2} \le \frac{1}{\sqrt{2\pi}}\int_{\lambda-t}^{\lambda+t}dx = \sqrt{\frac{2}{\pi}}e^{-\frac{\paren{0\lor\paren{\lambda-t}}^2}{2}}t. \label{eq:chi-square-cdf-bound}
    \end{align}

    We verify that $X_i$ satisfies the small-ball condition by utilizing \cref{eq:chi-square-cdf-bound}. Recall that $\abrace{\theta, X_i}^2/\sigma^2_X\norm{\theta}^2$ follows the non-central $\chi^2$ distribution with degree of freedom $1$ and non-centrality parameter $\lambda^2 = \abrace{\theta, \mu}^2/\sigma^2_X\norm{\theta}^2$ for any $\theta \in \RealSet^d\setminus\cbrace{0}$. By \cref{eq:chi-square-cdf-bound}, we have
    \begin{align}
        \p\cbrace*{\frac{\abs*{\abrace{\theta, X_i}}}{\sigma^2_X\norm{\theta}^2}^2 < t^2} \le& \sqrt{\frac{2}{\pi}}e^{-\frac{\paren{0\lor\paren{\lambda-t}}^2}{2}}t.
    \end{align}
    Noting that $\norm*{\Sigma^{\nicefrac{1}{2}}\theta}^2 = \sigma^2_X\norm{\theta}^2 + \abrace{\theta,\mu}^2$, we have
    \begin{align}
        \p\cbrace*{\abrace*{\theta, X_i}^2 \le t^2\norm*{\Sigma^{\nicefrac{1}{2}}\theta}^2} \le& \sqrt{\frac{2}{\pi}\paren*{1 + \frac{\abrace*{\theta,\mu}^2}{\sigma^2_X\norm{\theta}^2}}}e^{-\frac{\paren{0\lor\paren{\lambda-t}}^2}{2}}t \\
        =& \sqrt{\frac{2}{\pi}\paren*{1 + \lambda^2}}e^{-\frac{\paren{0\lor\paren{\lambda-t}}^2}{2}}t. \label{eq:min-eigvalue-sig-1}
    \end{align}
    We divide into two cases, $\lambda > t$ and $\lambda \le t$, to derive an upper bound on \cref{eq:min-eigvalue-sig-1}.

    (Case $\lambda > t$) 
    Since $(\lambda - t)^2 = \lambda^2 - 2\lambda t + t^2 \ge \lambda ^2 - 2t^2 + t^2 = \lambda^2 - t^2$, an upper bound on \cref{eq:min-eigvalue-sig-1} is obtained as
    \begin{align}
        & \sqrt{\frac{2}{\pi}\paren*{1 + \lambda^2}}e^{-\frac{\paren{0\lor\paren{\lambda-t}}^2}{2}}t \\
        \le& \sqrt{\frac{2}{\pi}\paren*{1 + \lambda^2}}e^{-\frac{\lambda^2 - t^2}{2}}t.
    \end{align}
    For positive numbers $a$ and $b$, $\sqrt{a+b} \le \sqrt{a} + \sqrt{b}$. Using this fact, we have
    \begin{align}
        & \sqrt{\frac{2}{\pi}\paren*{1 + \lambda^2}}e^{-\frac{\paren{0\lor\paren{\lambda-t}}^2}{2}}t \\
        \le& \sqrt{\frac{2}{\pi}}\paren*{\sqrt{1 + t^2} + \sqrt{\lambda^2 - t^2}e^{-\frac{\lambda^2 - t^2}{2}}}t.
    \end{align}
    Since a function $x \to xe^{-\nicefrac{x^2}{2}}$ admits a maximum on $x \in (0,\infty)$ of $e^{-\nicefrac{1}{2}}$, we have $\sqrt{\lambda^2 - t^2}e^{-\frac{\lambda^2 - t^2}{2}} \le e^{-\nicefrac{1}{2}}$. Consequently, we have
    \begin{align}
        \sqrt{\frac{2}{\pi}\paren*{1 + \lambda^2}}e^{-\frac{\paren{0\lor\paren{\lambda-t}}^2}{2}}t 
        \le \sqrt{\frac{16}{\pi e}\paren*{1 + t^2}}t, \label{eq:lambda-case1}
    \end{align}
    where we use the fact $(1+e^{-\nicefrac{1}{2}})^2 \le 8e^{-1}$.

    (Case $\lambda \le t$) We can easily verify that 
    \begin{align}
        \sqrt{\frac{2}{\pi}\paren*{1 + \lambda^2}}e^{-\frac{\paren{0\lor\paren{\lambda-t}}^2}{2}}t \le \sqrt{\frac{2}{\pi}\paren*{1 + t^2}}t \label{eq:lambda-case2}
    \end{align}

    Combining \cref{eq:lambda-case1,eq:lambda-case2}, we have for every $t > 0$,
    \begin{align}
        \p\cbrace*{\abrace*{\theta, X_i}^2 \le t^2\norm*{\Sigma^{\nicefrac{1}{2}}\theta}^2} \le \sqrt{\frac{16}{\pi e}\paren*{1 + t^2}}t.
    \end{align}
    For $t^2 \in (0, -\frac{1}{2} + \frac{1}{2}\sqrt{1 + \frac{\pi e}{4}}]$, we have
    \begin{align}
        \sqrt{\frac{16}{\pi e}\paren*{1 + t^2}}t \le \sqrt{\frac{8}{\pi e}\paren*{1 + \sqrt{1 + \frac{\pi e}{4}}}}t.
    \end{align}
    For $t \ge -\frac{1}{2} + \frac{1}{2}\sqrt{1 + \frac{\pi e}{4}}$, 
    \begin{align}
        \sqrt{\frac{8}{\pi e}\paren*{1 + \sqrt{1 + \frac{\pi e}{4}}}}t \ge \sqrt{\frac{4}{\pi e}\paren*{\paren*{1 + \frac{\pi e}{4}} - 1}} = 1.
    \end{align}
    Hence, for every $t > 0$, we have
    \begin{align}
        \p\cbrace*{\abrace*{\theta, X_i}^2 \le t^2\norm*{\Sigma^{\nicefrac{1}{2}}\theta}^2} \le \sqrt{\frac{8}{\pi e}\paren*{1 + \sqrt{1 + \frac{\pi e}{4}}}}t \le 7^{\nicefrac{1}{4}}t.
    \end{align}
    It confirms $X_i$ satisfies the small-ball condition with $C=7^{\nicefrac{1}{4}}$ and $\alpha = 1$. Application of \cref{thm:eig-inq-tail} yields the desired claim.
\end{proof}
\begin{proof}[Proof of \cref{lem:exp-max-eig-inv}]
    For a positive random variable $X$, we can express the expected value of $X$ as $\Mean\bracket{X} = \int_0^\infty \p\cbrace{X > t}dt$. Applying this to our problem, we obtain
    \begin{align}
        \Mean\bracket*{\lambda_{\max}\paren*{\paren*{\frac{1}{n}X^{\top}X}^{-1}}} =& \int_0^\infty \p\cbrace*{\lambda_{\max}\paren*{\paren*{\frac{1}{n}X^{\top}X}^{-1}} > t} dt \\
        =& \int_0^\infty \p\cbrace*{\lambda_{\min}\paren*{\paren*{\frac{1}{n}X^{\top}X}} < t^{-1}} dt.
    \end{align}
    Now, let us set $C = \frac{21e^{10}}{\sigma^2_X}$. Using the previous result, we can rewrite the expectation of interest as
    \begin{align}
        \Mean\bracket*{\lambda_{\max}\paren*{\paren*{\frac{1}{n}X^{\top}X}^{-1}}} =& C + \int_C^\infty \p\cbrace*{\lambda_{\min}\paren*{\paren*{\frac{1}{n}X^{\top}X}} < t^{-1}} dt \\
        \le& C + \int_C^\infty \paren*{C t^{-1}}^{\nicefrac{n}{6}} dt \\
        =& C + C^{\nicefrac{n}{6}} \paren*{1 - \frac{n}{6}}^{-1}\paren*{-C^{1-\nicefrac{n}{6}}} \\
        =& C\paren*{1 + \frac{6}{n - 6}},
    \end{align}
    which yields the claim.
\end{proof}

\parahead{Proofs of theorems} By utilizing the results of \cref{lem:ell-sph,lem:ell-angle,lem:min-eigvalue-sig,lem:exp-max-eig-inv}, we provide the complete proofs for both \cref{thm:err-unit-beta} and \cref{thm:err-norm}.
\begin{proof}[Proof of \cref{thm:err-unit-beta}]
 We begin by demonstrating that obtaining an upper bound on the expected error of the direction estimator can be reduced to finding a lower bound on the inner product $A_n = \abrace{\frac{\hat\beta}{\norm{\hat\beta}},\frac{\beta^*}{\norm{\beta^*}}}$. Specifically, a straightforward calculation yields
  \begin{align}
    \Mean\bracket*{\norm*{\frac{\hat\beta}{\norm{\hat\beta}} - \frac{\beta^*}{\norm{\beta^*}}}^2} =& 2\paren*{1 - \Mean\bracket*{\abrace*{\frac{\hat\beta}{\norm{\hat\beta}},\frac{\beta^*}{\norm{\beta^*}}}}}. \label{eq:err-unit-a-n}
  \end{align}
  Therefore, it suffices to establish a lower bound on $\Mean\bracket{A_n}$.

  Taking advantage of \cref{lem:ell-angle}, we derive a lower bound on $A_n$. Let $r = \paren{\hat\beta - \beta^*}^{\top}\paren{\frac{n}{\sigma^2_\xi}\Sigma_n}\paren{\hat\beta - \beta^*}$. From \cref{lem:ell-angle}, it follows that
  \begin{align}
      A_n \ge \sqrt{1 - \frac{\sigma^2_\xi r}{\norm{\beta^*}^2n}\lambda_{\max}\paren*{\Sigma_n^{-1}}},
  \end{align}
  provided that $r \le n\norm{\beta^*}^2 / \sigma^2_\xi\lambda_{\min}\paren{\Sigma_n^{-1}}$. Since $1 - \sqrt{1 - x} \le x$ for $x \in [0,1]$, it follows that
  \begin{align}
      1 - A_n \le \frac{\sigma^2_\xi r}{\norm{\beta^*}^2n}\lambda_{\max}\paren*{\Sigma_n^{-1}},
  \end{align}
  as long as $r \le n\norm{\beta^*}^2 / \sigma^2_\xi\lambda_{\max}\paren{\Sigma_n^{-1}}$.

  Next, we derive an upper bound on the expectation of $1-A_n$. Noting that conditioned on $X$, $r$ follows the $\chi^2$ distribution with degree of freedom $d$, we have
  \begin{align}
      &\Mean\bracket*{1 - A_n \middle| X} \\
      =& \Mean\bracket*{\paren*{\ind\cbrace*{r \le \frac{n\norm*{\beta^*}^2}{\sigma^2_\xi\lambda_{\max}\paren{\Sigma_n^{-1}}}} + \ind\cbrace*{r > \frac{n\norm*{\beta^*}^2}{\sigma^2_\xi\lambda_{\max}\paren{\Sigma_n^{-1}}}}}\paren*{1-A_n}\middle|X} \\
      \le& \Mean\bracket*{\frac{\sigma^2_\xi r}{\norm{\beta^*}^2n}\lambda_{\max}\paren*{\Sigma_n^{-1}} \middle| X} + \p\cbrace*{r > \frac{n\norm*{\beta^*}^2}{\sigma^2_\xi\lambda_{\max}\paren{\Sigma_n^{-1}}} \middle| X} \\
      \le& 2\Mean\bracket*{\frac{\sigma^2_\xi r}{\norm{\beta^*}^2n}\lambda_{\max}\paren*{\Sigma_n^{-1}} \middle| X} \label{eq:err-unit-beta-markov}\\
      =& \frac{2\sigma^2_\xi d}{\norm{\beta^*}^2n}\lambda_{\max}\paren*{\Sigma_n^{-1}}, \label{eq:err-unit-beta-1-a}
  \end{align}
  where we use the Markov inequality to obtain \cref{eq:err-unit-beta-markov}. 

 By utilizing \cref{eq:err-unit-a-n}, an upper bound on the expected error can be obtained by deriving an upper bound on the expectation of \cref{eq:err-unit-beta-1-a}. The random variable in \cref{eq:err-unit-beta-1-a} is $\lambda_{\max}\paren*{\Sigma_n^{-1}}$, which allows us to derive the upper bound on the expected error by obtaining an upper bound on the expectation of $\lambda_{\max}\paren*{\Sigma_n^{-1}}$. To accomplish this, we apply \cref{lem:exp-max-eig-inv}. The upper bound from \cref{lem:exp-max-eig-inv} can be substituted into \cref{eq:err-unit-beta-1-a}, resulting in the claimed upper bound.
\end{proof}
\begin{proof}[Proof of \cref{thm:err-norm}]
  We first utilize \cref{lem:ell-sph} to get an upper bound on the squared error of the norm estimator. Let $r = \paren{\hat\beta - \beta^*}^{\top}\paren{\frac{n}{\sigma^2_\xi}\Sigma_n}\paren{\hat\beta - \beta^*}$. From \cref{lem:ell-sph}, we have 
  \begin{align}
      \norm*{\hat\beta - \beta^*}^2 \le \frac{\sigma^2_\xi r}{n}\lambda_{\max}\paren*{\Sigma_n^{-1}}.
  \end{align}
  Application of the triangle and reverse triangle inequality yields
  \begin{align}
      \norm*{\beta^*} - \sqrt{\frac{\sigma^2_\xi r}{n}\lambda_{\max}\paren*{\Sigma_n^{-1}}} \le \norm*{\hat\beta} \le \norm*{\beta^*} + \sqrt{\frac{\sigma^2_\xi r}{n}\lambda_{\max}\paren*{\Sigma_n^{-1}}},
  \end{align}
  equivalently
  \begin{align}
      \paren*{\norm*{\hat\beta} - \norm*{\beta^*}}^2 \le \frac{\sigma^2_\xi r}{n}\lambda_{\max}\paren*{\Sigma_n^{-1}}.
  \end{align}

  Taking expectation conditioned on $X$ yields
  \begin{align}
    & \Mean\bracket*{\paren*{\norm{\hat\beta} - \norm{\beta^*}}^2 \middle| X} \\
    =& \Mean\bracket*{\frac{\sigma^2_\xi r}{n}\lambda_{\max}\paren*{\Sigma_n^{-1}} \middle| X} \\
    \le& \frac{\sigma^2_\xi d}{n}\lambda_{\max}\paren*{\Sigma_n^{-1}}, \label{eq:err-norm-1}
  \end{align}
  where we use the fact that $r$ follows the $\chi^2$ distribution with degree of freedom $d$ to obtain the last line. Again, application of \cref{lem:exp-max-eig-inv} into expectation of \cref{eq:err-norm-1} yields the claim.
\end{proof}

\section{Details of Upper Bound Analyses}
This section presents a detailed proof of the upper bound stated in \cref{thm:main-err}, which is achieved through an analysis of the estimator constructed in  \cref{sec:est}. Specifically, we establish the following theorem:
\begin{theorem}\label{thm:loose-upper}
 Let $\hat\beta_n$ be the estimator constructed in \cref{sec:est}. Then, there exists a universal constant $C > 0$ such that for any $\delta \in (0,1)$ and $n \ge 12(3d\lor4\ln(M/\delta))/\min_{s\in[M]}p_s$, 
 \begin{align}
     \mathcal{E}_{n}\paren*{\frac{1}{2},\delta} \le C\frac{\sigma^2_\xi B^2 dM \lor \sigma^2_XB^2M \lor B^2U^2}{n} + o\paren*{\frac{1}{n}}.
 \end{align}
\end{theorem}

To establish the validity of \cref{thm:loose-upper}, we begin by proving \cref{thm:decompose}, which demonstrates that the estimation error can be decomposed into the sum of errors associated with individual components. Subsequently, we derive upper bounds for the estimation errors of each component. Finally, we synthesize these results to provide a proof of \cref{thm:loose-upper}.

\subsection{Proof of \crtcref{thm:decompose}}
We commence the error analysis of our estimator by decomposing the estimation error, as presented in \cref{thm:decompose}. Recall the statement of \cref{thm:decompose}
\begin{theorem}\label{thm:decompose-a}
 For the estimator defined in \cref{eq:loose-reg}, the mean square deviation from $f^*_{\mathrm{DP}}$ is bounded above by
 \begin{multline}
      \sum_{s \in [M]}p_s\Mean\bracket[\Bigg]{\paren[\Bigg]{\Mean\bracket*{\widehat{\norm{\beta_\cdot}}^2\middle|n_\cdot}^{\nicefrac{1}{2}}\Mean\bracket*{\abrace*{\tilde\beta_s,\mu_s-\hat\mu_{s}}^2\middle|n_\cdot}^{\nicefrac{1}{2}} + \sigma_X\Mean\bracket*{\paren*{\widehat{\norm{\beta_\cdot}} - \overline{\norm*{\beta^*_\cdot}}}^2\middle|n_\cdot}^{\nicefrac{1}{2}} + \\ \sigma_X\overline{\norm{\beta^*_\cdot}}\Mean\bracket*{\norm*{\tilde\beta_s - {\beta^*_s}/{\norm{\beta^*_s}}}^2\middle|n_\cdot}^{\nicefrac{1}{2}} + \Mean\bracket*{\paren*{\sum_{s' \in [M]}\hat{p}_{s'}\abrace*{\hat\beta'_{s'} - \beta^*_{s'},\hat\mu'_{s'}}}^2\middle|n_\cdot}^{\nicefrac{1}{2}} + \\ \Mean\bracket*{\paren*{\sum_{s' \in [M]}\hat{p}_{s'}\abrace*{\beta^*_{s'},\hat\mu'_{s'} - \mu_{s'}}}^2\middle|n_\cdot}^{\nicefrac{1}{2}} + \abs*{\sum_{s' \in [M]}\paren*{\hat{p}_{s'} - p_{s'}}\abrace*{\beta^*_{s'},\mu_{s'}}}}^2}. 
 \end{multline}
\end{theorem}
We provide a proof of \cref{thm:decompose-a} as follows:
\begin{proof}[Proof of \cref{thm:decompose-a}]
  We begin by decomposing $\hat{f}_n(X,S) - f^*_{\mathrm{DP}}(X,S)$ into six terms. Recall the definitions of $\hat{f}_n(x,s)$ and $f^*_{\mathrm{DP}}(x,s)$:
  \begin{align}
      \hat{f}_n(x,s) =&  \widehat{\norm{\beta_{\cdot}}}\abrace*{\tilde\beta_s, x - \hat\mu_s} + \sum_{s'\in[M]}\hat{p}_{s'}\abrace*{\hat\beta'_s,\hat\mu'_s} \\
      f^*_{\mathrm{DP}}(x,s) =& \overline{\norm{\beta^*_\cdot}}\abrace*{\frac{\beta^*_s}{\norm{\beta^*_s}},x - \mu_s} + \smashoperator{\sum{s'\in[M]}}p_{s'}\abrace*{\beta^*_{s'},\mu_{s'}}.
  \end{align}
  Through elementary calculations, we obtain:
  \begin{multline}
      \hat{f}_n(X,S) - f^*_{\mathrm{DP}}(X,S) = \widehat{\norm{\beta{\cdot}}}\abrace*{\tilde\beta_S, \mu_S - \hat\mu_S} + \paren*{\widehat{\norm{\beta_{\cdot}}} - \norm{\beta^*_{\cdot}}}\abrace*{\tilde\beta_S, X - \mu_S} \\ +  \norm{\beta^*_{\cdot}}\abrace*{\tilde\beta_S - \frac{\beta^*_S}{\norm{\beta^*_S}}, X - \mu_S} + \sum_{s' \in [M]}\hat{p}_{s'}\abrace*{\hat\beta'_{s'} - \beta^*_{s'}, \hat\mu'_{s'}} \\ + \sum_{s' \in [M]}\hat{p}_{s'}\abrace*{\beta^*_{s'}, \hat\mu'_{s'} - \mu_{s'}} + \sum_{s' \in [M]}\paren*{\hat{p}_{s'} - p_{s'}}\abrace*{\beta^*_{s'}, \mu_{s'}}. \label{eq:decompose-dec}
  \end{multline}
  By the Cauchy-Schwarz inequality, for two random variable $Z_1$ and $Z_2$, we have $\Mean\bracket{(Z_1 + Z_2)^2}^{\nicefrac{1}{2}} \le \Mean\bracket{Z_1^2}^{\nicefrac{1}{2}} + \Mean\bracket{Z_2^2}^{\nicefrac{1}{2}}$. By applying this fact into the expectation of \cref{eq:decompose-dec} conditioned on $S$ and $n_{\cdot}$ multiple times, we have
  \begin{align}
      & \Mean\bracket*{\paren*{\hat{f}_n(X,S) - f^*_{\mathrm{DP}}(X,S)}^2} \\
      =& \sum_{s \in [M]}p_s\Mean\bracket*{\Mean\bracket*{\paren*{\hat{f}_n(X,S) - f^*_{\mathrm{DP}}(X,S)}^2\middle|S=s,n_{\cdot}} \middle| S=s} \\
      \le& \begin{multlined}[t][.9\textwidth]
        \sum_{s \in [M]}p_s\Mean\bracket[\Bigg]{\paren[\Bigg]{
            \Mean\bracket*{\paren*{\widehat{\norm{\beta{\cdot}}}\abrace*{\tilde\beta_S, \mu_S - \hat\mu_S}}^2\middle|S=s,n_{\cdot}}^{\nicefrac{1}{2}} \\
            + \Mean\bracket*{\paren*{\paren*{\widehat{\norm{\beta_{\cdot}}} - \norm{\beta^*_{\cdot}}}\abrace*{\tilde\beta_S, X - \mu_S}}^2\middle|S=s,n_{\cdot}}^{\nicefrac{1}{2}} \\
            + \Mean\bracket*{\paren*{\norm{\beta^*_{\cdot}}\abrace*{\tilde\beta_S - \frac{\beta^*_S}{\norm{\beta^*_S}}, X - \mu_S}}^2\middle|S=s,n_{\cdot}}^{\nicefrac{1}{2}} \\
            + \Mean\bracket*{\paren*{\sum_{s' \in [M]}\hat{p}_{s'}\abrace*{\hat\beta'_{s'} - \beta^*_{s'}, \hat\mu'_{s'}}}^2\middle|S=s,n_{\cdot}}^{\nicefrac{1}{2}} \\
            + \Mean\bracket*{\paren*{\sum_{s' \in [M]}\hat{p}_{s'}\abrace*{\beta^*_{s'}, \hat\mu'_{s'} - \mu_{s'}}}^2\middle|S=s,n_{\cdot}}^{\nicefrac{1}{2}} \\
            + \Mean\bracket*{\paren*{\sum_{s' \in [M]}\paren*{\hat{p}_{s'} - p_{s'}}\abrace*{\beta^*_{s'}, \mu_{s'}}}^2\middle|S=s,n_{\cdot}}^{\nicefrac{1}{2}}
        }^2\Bigg|S=s}.
      \end{multlined} \label{eq:decompose-eq1}
  \end{align}
  In the subsequent analyses, we derive upper bounds for each term in \cref{eq:decompose-eq1}.

 (First term in \cref{eq:decompose-eq1}) Due to the splitting of the sample, $\widehat{\norm{\beta_\cdot}}$, $\tilde\beta_s$, and $\hat\mu_s$ are independent conditioned on $n_\cdot$. Thus, we have:
 \begin{align}
     & \Mean\bracket*{\paren*{\widehat{\norm{\beta{\cdot}}}\abrace*{\tilde\beta_S, \mu_S - \hat\mu_S}}^2\middle|S=s,n_{\cdot}}^{\nicefrac{1}{2}} \\
     =& \Mean\bracket*{\paren*{\widehat{\norm{\beta_\cdot}}\abrace*{\tilde\beta_s, \mu_s - \hat\mu_{s}}}^2\middle|n_\cdot}^{\nicefrac{1}{2}} \\
     =& \Mean\bracket*{\paren*{\widehat{\norm{\beta_\cdot}}^2}\middle|n_\cdot}^{\nicefrac{1}{2}}\Mean\bracket*{\abrace*{\tilde\beta_s,\mu_s-\hat\mu_{s}}^2\middle|n_\cdot}^{\nicefrac{1}{2}}.
 \end{align}
 This term matches the first term of the desired bound.

 (Second term in \cref{eq:decompose-eq1}) Since $\widehat{\norm{\beta_\cdot}}$, $\tilde\beta_s$, and $X$ are independent conditioned on $n_\cdot$, we have
 \begin{align}
     & \Mean\bracket*{\paren*{\paren*{\widehat{\norm{\beta_\cdot}} - \overline{\norm{\beta^*_\cdot}}}\abrace*{\tilde\beta_s,X - \mu_s}}^2\middle|S=s,n_\cdot} \\
     =& \sigma^2_X\Mean\bracket*{\paren*{\widehat{\norm{\beta_\cdot}} - \overline{\norm{\beta^*_\cdot}}}^2\middle|n_\cdot},
 \end{align}
 where we use the fact that $X-\mu_s \sim N(0,\sigma^2_XI)$ conditioned on $S=s$, and $\tilde\beta_s \in \dom{S}_{d-1}$ almost surely. This result corresponds to the second term of the desired bound.

 (Third term in \cref{eq:decompose-eq1}) Since $X-\mu_s \sim N(0,\sigma^2_XI)$, we have
 \begin{align}
     & \Mean\bracket*{\paren*{\overline{\norm{\beta^*_\cdot}}\abrace*{\tilde\beta_s - \frac{\beta^*_s}{\norm{\beta^*_s}},X - \mu_s}}^2\middle|S=s,n_\cdot} \\
     =& \sigma^2_X\overline{\norm{\beta^*_\cdot}}^2\Mean\bracket*{\norm*{\tilde\beta_s - \frac{\beta^*_s}{\norm{\beta^*_s}}}^2\middle|n_\cdot},
 \end{align}
 which corresponds to the third term of the desired bound.

 (Forth and fifth terms in \cref{eq:decompose-eq1}) These terms are independent of $S$, so we can omit $S=s$ from the condition, resulting in the fourth and fifth terms of the desired bound.

 (Sixth term in \cref{eq:decompose-eq1}) This term does not contain any random variable when $n_\cdot$ is fixed. Thus, we can remove the expectation, yielding the sixth term of the desired bound.
\end{proof}

\subsection{Estimation Error Analyses for Each Component}
This subsection presents an analysis of the estimation errors associated with each component estimator. In particular, we investigate the estimation errors of $\hat\mu_s$, $\widehat{\norm{\beta_\cdot}}$, $\tilde\beta_s$, $\hat\beta'_s$, and $\hat\mu'_s$.

\subsubsection{Estimation Error Analysis for $\hat\mu_s$}
Here, we presents the proof of the following theorem.
\begin{theorem}\label{thm:decomp-first}
 Given $s \in [M]$, if $n_s > 18d$, we have
 \begin{align}
     \sup_{v \in \mathbb{S}_{d-1}}\Mean\bracket*{\abrace*{v,\mu_s-\hat\mu_{s}}^2\middle|n_\cdot} \le \frac{6\sigma^2_X}{n_{s}}.
 \end{align}
\end{theorem}
\begin{proof}[Proof of \cref{thm:decomp-first}]
    Given $v \in \mathbb{S}_{d-1}$, we have
    \begin{align}
        \Mean\bracket*{\abrace*{v,\mu_s-\hat\mu_{s}}^2\middle|n_\cdot} = \abrace*{v, \Mean\bracket*{\paren*{\mu_s-\hat\mu_s}\paren*{\mu_s-\hat\mu_s}^\top\middle|n_\cdot}v}.
    \end{align}
    According to the definition, $\hat\mu_s$ is an average of $n_{3,s}$ i.i.d. random variables following $N(\mu_s,\sigma^2_XI)$. Hence, we have $\mu_s - \hat\mu_s \sim N(0, \frac{\sigma^2_X}{n_{3,s}}I)$, which implies $\Mean\bracket{\paren{\mu_s-\hat\mu_s}\paren{\mu_s-\hat\mu_s}^\top|n_\cdot} = \frac{\sigma^2_X}{n_{3,s}}I$. Consequently, we obtain:
    \begin{align}
        \Mean\bracket*{\abrace*{v,\mu_s-\hat\mu_{s}}^2\middle|n_\cdot} =& \abrace*{v, \frac{\sigma^2_X}{n_{3,s}}Iv} \\
        =& \frac{\sigma^2_X}{n_{3,s}}\abrace*{v, v} = \frac{\sigma^2_X}{n_{3,s}}.
    \end{align}
  Since $n_{3,s} \ge \floor{n_s/3} \ge n_s/6$ for $n_s \ge 6$, the claim follows.
\end{proof}

\subsubsection{Estimation Error Analysis for $\widehat{\norm{\beta_\cdot}}$}
Here, we present the proof of the following theorem.
\begin{theorem}\label{thm:decomp-second}
 For any $s \in [M]$, we have
 \begin{align}
     \Mean\bracket*{\paren*{\widehat{\norm{\beta_\cdot}} - \overline{\norm*{\beta^*_\cdot}}}^2\middle|n_\cdot}^{\nicefrac{1}{2}} \le \sqrt{\frac{189e^{10}\sigma^2_\xi Md}{\sigma^2_Xn} + \sum_{s\in[M]}\ind\cbrace*{n_s \le 18d}\frac{n_s}{n}} + \abs*{\sum_{s\in[M]}\paren*{\hat{p}_s - p_s}\norm*{\beta^*_s}}.
 \end{align}
\end{theorem}
\begin{proof}[Proof of \cref{thm:decomp-second}]
 By combining the definitions of $\widehat{\norm{\beta_\cdot}}$ and $\overline{\norm*{\beta^*_\cdot}}$ and utilizing the Cauchy-Schwarz inequality, we obtain
 \begin{align}
     & \Mean\bracket*{\paren*{\widehat{\norm{\beta_\cdot}} - \overline{\norm*{\beta^*_\cdot}}}^2\middle|n_\cdot}^{\nicefrac{1}{2}} \\
     =& \Mean\bracket*{\paren*{\sum_{s \in [M]}\hat{p}_s\paren*{\widehat{\norm{\beta_s}} - \overline{\norm*{\beta^*_s}}} + \sum_{s\in[M]}\paren*{\hat{p}_s - p_s}\overline{\norm*{\beta^*_s}}}^2\middle|n_\cdot}^{\nicefrac{1}{2}} \\
     \le& \Mean\bracket*{\paren*{\sum_{s \in [M]}\hat{p}_s\paren*{\widehat{\norm{\beta_s}} - \overline{\norm*{\beta^*_s}}}}^2\middle|n_\cdot}^{\nicefrac{1}{2}} + \abs*{\sum_{s\in[M]}\paren*{\hat{p}_s - p_s}\overline{\norm*{\beta^*_s}}}. \label{eq:decomp-second-1}
 \end{align}
 
 Next, we derive an upper bound for the first term in \cref{eq:decomp-second-1}. Applying Jensen's inequality, we have:
 \begin{align}
     & \Mean\bracket*{\paren*{\sum_{s \in [M]}\hat{p}_s\paren*{\widehat{\norm{\beta_s}} - \overline{\norm*{\beta^*_s}}}}^2\middle|n_\cdot} \\
     \le& \sum_{s \in [M]}\frac{n_s}{n}\Mean\bracket*{\paren*{\widehat{\norm{\beta_s}} - \overline{\norm*{\beta^*_s}}}^2\middle|n_\cdot} \\
     =& \sum_{s \in [M]}\frac{n_s}{n}\paren*{\ind\cbrace*{n_s > 18d}\Mean\bracket*{\paren*{\widehat{\norm{\beta_s}} - \overline{\norm*{\beta^*_s}}}^2\middle|n_\cdot} + \ind\cbrace*{n_s \le 18d}\norm{\beta^*_s}^2}
 \end{align}
 Using the fact that $n_{1,s} > 6d$ for $n_s > 18d$ and employing \cref{thm:err-norm}, we obtain:
 \begin{align}
     & \Mean\bracket*{\paren*{\sum_{s \in [M]}\hat{p}_s\paren*{\widehat{\norm{\beta_s}} - \overline{\norm*{\beta^*_s}}}}^2\middle|n_\cdot} \\
     \le& \sum_{s \in [M]}\frac{n_s}{n}\frac{21e^{10}\sigma^2_\xi d}{n_{1,s}}\paren*{1 + \frac{6}{n_{1,s} - 6}} + \sum_{s \in [M]}\ind\cbrace*{n_s \le 18d}\frac{\norm{\beta^*_s}^2n_s}{n} \\
     \le& \frac{189e^{10}\sigma^2_\xi Md}{\sigma^2_X n} + \sum_{s\in[M]}\ind\cbrace*{n_s \le 18d}\frac{B^2n_s}{n},
 \end{align}
 where the last line follows from the fact that $n_{1,s} \ge \floor{n_s/3} \ge n_s/6$ for $n_s \ge 6$, $\frac{6}{n_{1,s}-6} \le \nicefrac{1}{2}$ for $n_s \ge 18$, and $\norm{\beta^*_s} \le B$. 
\end{proof}

\subsubsection{Estimation Error Analysis for $\tilde\beta_s$}
Here, we will prove the following theorem.
\begin{theorem}\label{thm:decomp-forth}
 For any $s \in [M]$, we have
 \begin{align}
     \Mean\bracket*{\norm*{\tilde\beta_s - \frac{\beta^*_s}{\norm{\beta^*_s}}}^2\middle|n_\cdot} \le \begin{dcases}
      \frac{756e^{10}\sigma^2_\xi d}{\sigma^2_X\norm{\beta^*_s}^2n_{s}} & \textif n_s > 18d, \\
      1 & \otherwise.
     \end{dcases}
 \end{align}
\end{theorem}
\begin{proof}[Proof of \cref{thm:decomp-forth}]
 If $n_s \le 18d$, $\tilde\beta_s = 0$, and we thus have $\Mean\bracket{\norm{\tilde\beta_s - \frac{\beta^*_s}{\norm{\beta^*_s}}}^2|n_\cdot} = \norm{\frac{\beta^*_s}{\norm{\beta^*_s}}}^2 = 1$. For $n_s > 18d$, we have $n_{2,s} > 6d$. Application of \cref{thm:err-unit-beta} yields
 \begin{align}
     \Mean\bracket*{\norm*{\tilde\beta_s - \frac{\beta^*_s}{\norm{\beta^*_s}}}^2\middle|n_\cdot} \le \frac{84e^{10}\sigma^2_\xi d}{\sigma^2_X\norm{\beta^*_s}^2n_{2,s}}\paren*{1 + \frac{6}{n_{2,s}-6}}
 \end{align}
 We get the claim in the same manner as the proof of \cref{thm:decomp-second}.
\end{proof}

\subsubsection{Estimation Error Analysis for $\hat\beta'_{s'}$}
Here, we will prove the following theorem.
\begin{theorem}\label{thm:decomp-fifth}
  Given $s \in [M]$, let $\Sigma_{s} = \Mean\bracket{X_sX_s^\top}$ for $X_s \sim N(\mu_s,\sigma^2_XI)$. Then, if $n_s > 12d$, we have
 \begin{align}
     \Mean\bracket*{\norm*{\Sigma_{s}^{\nicefrac{1}{2}}\paren*{\hat\beta'_{s} - \beta^*_{s}}}^2\middle|n_\cdot} \le \frac{4\sigma^2_\xi d}{n_{s}} + 504e^{10}\sigma^2_\xi\paren*{\frac{4d}{n_{s}}}^2.
 \end{align}
\end{theorem}
To prove \cref{thm:decomp-fifth}, we utilize the following theorem presented by \textcite{Mourtada2022ExactMatrices}. 
\begin{theorem}[Theorem 3 in \autocite{Mourtada2022ExactMatrices}]\label{thm:tr-err}
    Let $X$ be a random vector in $\RealSet^d$ such that it statisfies the small-ball condition of \cref{eq:small-ball} and $\Mean\bracket{\norm{\Sigma^{-\nicefrac{1}{2}}X}^4} \le \kappa d$ for some $\kappa > 0$, where $\Sigma=\Mean\bracket{XX^\top}$. Let $\hat\Sigma_n = \frac{1}{n}\sum_{i=1}^nX_iX_i^\top$, where $X_i$ are i.i.d. copies of $X$. If $n \ge 6\alpha^{-1}d\land12\alpha^{-1}\ln\paren{12\alpha^{-1}}$, 
    \begin{align}
        \frac{1}{n}\Mean\bracket*{\Tr\paren*{\hat\Sigma_n^{-1}\Sigma}} \le \frac{d}{n} + 8C'\kappa\paren*{\frac{d}{n}}^2,
    \end{align}
    where $\alpha$ and $C'$ are as in \cref{thm:eig-inq-tail}.
\end{theorem}
\begin{proof}[Proof of \cref{thm:decomp-fifth}]
We can easily confirm that $\hat\beta'_s \sim N(\beta^*_s, \frac{\sigma^2_\xi}{n'_{1,s}}(\Sigma'_{1,s})^{-1})$ conditioned on $n_\cdot$ and $X'_{1,s}$, where $\Sigma'_{1,s} = \frac{1}{n'_{1,s}}X'_{1,s}(X'_{1,s})^\top$. Noting that $\Mean\bracket{(\hat\beta'_s - \beta^*_s)(\hat\beta'_s - \beta^*_s)^\top|X,n_\cdot} = \frac{\sigma^2_\xi}{n'_{1,s}}(\Sigma'_{1,s})^{-1}$, we have
\begin{align}
    & \Mean\bracket*{\norm*{\Sigma_{m,s}^{\nicefrac{1}{2}}\paren*{\hat\beta'_{s'} - \beta^*_{s}}}^2\middle|n_\cdot} \\
    =& \Tr\paren*{\Sigma_{s} \Mean\bracket*{\paren*{\hat\beta'_{s'} - \beta^*_{s}}\paren*{\hat\beta'_{s'} - \beta^*_{s}}^\top \middle|n_\cdot}} \\
    =& \frac{\sigma^2_\xi}{n'_{1,s}}\Mean\bracket*{\Tr\paren*{\Sigma_{s} \paren*{\Sigma'_{1,s}}^{-1}}\middle|n_\cdot}. \label{eq:bd-tr}
\end{align}
We apply \cref{thm:tr-err} to the expected trace term in \cref{eq:bd-tr}. To do so, we need to check $X'_{1,s}$ satisfies the small-ball condition of \cref{eq:small-ball} and the kurtosis condition $\Mean\bracket{\norm{\Sigma_s^{-\nicefrac{1}{2}}X'_{1,s}}^4} \le \kappa d$. 

The small-ball condition is confirmed by the same manner in the proof of \cref{lem:exp-max-eig-inv}, with $C = 7^{\nicefrac{1}{4}}$ and $\alpha = 1$. Here, we prove the satisfication of the kurtosis condition.  For a multivariate Gaussian random variable $X \sim N(\mu,\Lambda)$ such that $\lambda_{\min}(\Lambda) > 0$, $\Sigma \coloneqq \Mean\bracket{XX^\top} = \mu\mu^\top + \Lambda$, and $\norm{\Sigma^{-\nicefrac{1}{2}}X}^4 = \abrace{X, \Sigma^{-1}X}$ 
\begin{align}
    \Mean\bracket*{\norm*{\Sigma^{-\nicefrac{1}{2}}X}^4} =& \Mean\bracket*{\abrace{X, \Sigma^{-1}X}^2} \\
    =& \Var\bracket*{\abrace{X, \Sigma^{-1}X}} + \paren*{\Mean\bracket*{\abrace{X, \Sigma^{-1}X}}}^2.
\end{align}
Since we have
\begin{align}
    \Mean\bracket*{\abrace{X, \Sigma^{-1}X}} =& \Tr\paren*{\Sigma^{-1}\paren*{\mu\mu^\top + \Lambda}} = d \\
    \Var\bracket*{\abrace{X, \Sigma^{-1}X}} =& 2\Tr\paren*{\Sigma^{-1}\Lambda\Sigma^{-1}\paren*{\mu\mu^\top + \Lambda}} + 2\Tr\paren*{\Sigma^{-1}\Lambda\Sigma^{-1}\mu\mu^\top} \\
    =& 2\Tr\paren*{\Sigma^{-1}\Lambda} + 2\Tr\paren*{\Sigma^{-1}\Lambda\Sigma^{-1}\mu\mu^\top} \\
    =& 2\Tr\paren*{\Sigma^{-1}\paren*{\mu\mu^\top + \Lambda}} - 2\Tr\paren*{\mu\mu^\top\Sigma^{-1}\mu\mu^\top} \\
    =& 2d - 2\norm{\mu}^2\abrace*{\mu, \Sigma^{-1}\mu} \le 2d.
\end{align}
Hence, the kurtosis condition satisfies with $\kappa = 3$.

Application of \cref{thm:tr-err} into \cref{eq:bd-tr} yields
\begin{align}
    \Mean\bracket*{\norm*{\Sigma_{m,s}^{\nicefrac{1}{2}}\paren*{\hat\beta'_{s'} - \beta^*_{s}}}^2\middle|n_\cdot} \le \frac{\sigma^2_\xi d}{n'_{1,s}} + 504e^{10}\sigma^2_\xi\paren*{\frac{d}{n'_{1,s}}}^2,
\end{align}
provided that $n_{s} \ge 12d$. We get the claim from the fact that $n'_{1,s} \ge \floor{n_s/2} \ge n_s/4$ for $n_s \ge 4$,. 
\end{proof}

\subsubsection{Estimation Error Analysis for $\hat\mu'_{s}$}
Here, we will prove the following theorem.
\begin{theorem}\label{thm:decomp-sixth}
  Given $s \in [M]$ and $v \in \RealSet^d$, if $n_s > 12d$, we have
 \begin{align}
     \Mean\bracket*{\abrace*{v,\hat\mu'_{s} - \mu_{s}}^2\middle|n_\cdot} \le  \frac{4\sigma^2_X\norm{v}^2}{n_{s}}.
 \end{align}
\end{theorem}
\begin{proof}[Proof of \cref{thm:decomp-sixth}]
    By definition, we have $\hat\mu'_s \sim N(\mu_s,\frac{\sigma^2_X}{n'_{2,s}}I)$ conditioned on $n_\cdot$ for $n_s > 12d$. Hence, we have
    \begin{align}
        \Mean\bracket*{\abrace*{v,\hat\mu'_{s} - \mu_{s}}^2\middle|n_\cdot} \le \frac{\sigma^2_X\norm{v}^2}{n'_{2,s}}.
    \end{align}
    We get the claim following the same manner of the proof of \cref{thm:decomp-fifth}.
\end{proof}

\subsection{Some Auxiliary Lemmas}
This subsections introduce some auxiliary lemmas for use to prove \cref{thm:loose-upper}. Specifically, we demonstrate the following lemmas:
\begin{lemma}\label{lem:diff-p}
    Let $a_1,...,a_M \in \RealSet$ be arbitrary numbers. Then, we have
    \begin{align}
        \Mean\bracket*{\paren*{\sum_{s \in [M]}a_s\paren*{\hat{p}_s - p_s}^2}} = \frac{1}{n}\Var\bracket*{a_S}.
    \end{align}
\end{lemma}
\begin{lemma}\label{lem:inv-n}
    For a constant $c > 0$, we have for any $s \in [M]$
    \begin{align}
        \Mean\bracket*{n_s^{-1}\ind\cbrace*{n_s > c}} \le \frac{1 + c^{-1}}{p_s(n+1)}.
    \end{align}
\end{lemma}
\begin{lemma}\label{lem:small-n}
    Let $c > 0$ be a constant. If $n > 2c/\min_{s\in[M]}p_s$,we have for any $s \in [M]$
    \begin{align}
        \p\cbrace*{n_s \le c} \le e^{-np_s/8}.
    \end{align}
\end{lemma}
\begin{proof}[Proof of \cref{lem:diff-p}]
    Since $n\hat{p}_\cdot$ follows the multinomial distribution with the parameters $n$ and $p_\cdot$, using the variance and covariance of the multinomial distribution, we have
     \begin{align}
         & \Mean\bracket*{\paren*{\sum_{s\in[M]}a_s\paren*{\hat{p}_s-p_s}}^2} \\
         =& \sum_{s \in [M]}\frac{a^2_sp_s(1-p_s)}{n} - \sum_{s,s'\in[M]:s \ne s'}\frac{a_sa_{s'}p_sp_{s'}}{n} \\
         =& \sum_{s \in [M]}\frac{a_sp_s}{n}\paren*{a_s - \sum_{s'\in[M]}a_{s'}p_{s'}}.
     \end{align}
     Let $\bar{a} = \sum_{s\in[M]}p_sa_s$. Then, we have
     \begin{align}
         & \sum_{s \in [M]}a_sp_s\paren*{a_s - \bar{a}} \\
         =&  \sum_{s \in [M]}p_s\paren*{a_s - \bar{a}}^2 + \sum_{s\in[M]}\bar{a}p_s\paren*{a_s - \bar{a}} \\
         =& \sum_{s \in [M]}p_s\paren*{a_s - \bar{a}}^2 = \Var\bracket*{a_S}.
     \end{align}
     Hence, 
     \begin{align}
         \Mean\bracket*{\paren*{\sum_{s\in[M]}a_s\paren*{\hat{p}_s-p_s}}^2} = \frac{1}{n}\Var\bracket*{a_S}. \label{eq:error-hat-p}
     \end{align}
\end{proof}
\begin{proof}[Proof of \cref{lem:inv-n}]
For a random variable $X$ following the binomial distribution with the parameters $n$ and $p$, $\Mean\bracket{\frac{1}{X+1}}=\frac{1}{p(n+1)}(1-(1-p)^{n+1})$~\autocite{Chao1972NegativeVariables}. Since $n_s$ follows the binomial distribution with the parameters $n$ and $p_s$, we have
 We have
 \begin{align}
     & \Mean\bracket*{n_s^{-1}\ind\cbrace*{n_s > c}} \\
     =& \Mean\bracket*{\frac{1}{n_s+1}\paren*{1 + \frac{1}{n_s}}\ind\cbrace*{n_s > c}} \\
     \le& \paren*{1 + c^{-1}}\Mean\bracket*{\frac{1}{n_s+1}} \\
     =& \frac{1 + c^{-1}}{p_s(n+1)}\paren*{1 - (1-p_s)^{n+1}} \le \frac{1 + c^{-1}}{p_s(n+1)}. \label{eq:inv-size}
 \end{align}
\end{proof}
\begin{proof}[Proof of \cref{lem:small-n}]
 From the Chernoff bound, we have 
 \begin{align}
     \p\cbrace*{n_s \le c} \le \exp\paren*{-\frac{n}{2p_s}\paren*{\frac{c}{n} - p_s}^2}.
 \end{align}
 Under the assumption, we have $c \le np_s/2$. Then, we have $\frac{n}{2p_s}\paren*{\frac{c}{n} - p_s}^2 \ge n/8$, which gives the claim.
\end{proof}

\subsection{Proof of \crtcref{thm:loose-upper}}
\begin{proof}[Proof of \cref{thm:loose-upper}]
 We begin by characterizing the estimation error by each component's estimation error shown in \cref{thm:decomp-first,thm:decomp-second,thm:decomp-forth,thm:decomp-fifth,thm:decomp-sixth}. Specifically, we characterize the estimation error using the following error terms:
 \begin{align}
     e^2_{\mathrm{mean},s} =& \sup_{v \in \mathbb{S}_{d-1}} \Mean\bracket*{\abrace*{v,\mu_s - \hat\mu_s}^2\middle| n_\cdot} \\
     e^2_{\mathrm{norm}} =& \Mean\bracket{\paren{\overline{\norm{\beta^*_\cdot}} - \widehat{\norm*{\beta_\cdot}}}^2| n_\cdot} \\
     e^2_{\mathrm{coef},s} =& \Mean\bracket{\norm{\tilde\beta_s - \frac{\beta^*_s}{\norm{\beta^*_s}}}^2|n_\cdot} \\
     e^2_{\mathrm{coef}',s} =& \Mean\bracket*{\norm*{\Sigma_{s}^{\nicefrac{1}{2}}\paren*{\hat\beta'_{s} - \beta^*_{s}}}^2\middle|n_\cdot} \\
     e^2_{\mathrm{mean}',s} =& \sup_{v \in \mathbb{S}_{d-1}} \Mean\bracket*{\abrace{\hat\mu'_s - \mu_s, v}^2\middle|n_\cdot} \\
     e^2_{\mathrm{prob}} =& \abs*{\sum_{s' \in [M]}\paren*{\hat{p}_{s'} - p_{s'}}\abrace*{\beta^*_{s'},\mu_{s'}}}^2,
 \end{align}
 where $\Sigma_x = \Mean\bracket{X_sX_s^\top}$ for $X_s \sim N(\mu_s,\sigma^2_XI)$.

 We analyze each term in \cref{eq:decompose} one by one.

 (First term in \cref{thm:decompose}) Recall the first term in \cref{thm:decompose}
 \begin{align}
     \Mean\bracket*{\widehat{\norm*{\beta_\cdot}}^2\middle| n_\cdot}^{\nicefrac{1}{2}}\Mean\bracket*{\abrace*{\tilde\beta_s,\mu_s - \hat\mu_s}^2\middle| n_\cdot}^{\nicefrac{1}{2}}.
 \end{align}
 From the Cauchy–Schwarz inequality, we have
 \begin{align}
     \Mean\bracket*{\widehat{\norm*{\beta_\cdot}}^2\middle| n_\cdot}^{\nicefrac{1}{2}} \le& \Mean\bracket*{\overline{\norm*{\beta^*_\cdot}}^2\middle| n_\cdot}^{\nicefrac{1}{2}} + \Mean\bracket*{\paren*{\overline{\norm{\beta^*_\cdot}} - \widehat{\norm*{\beta_\cdot}}}^2\middle| n_\cdot}^{\nicefrac{1}{2}} \\
     \le& B + e_{\mathrm{norm}}.
 \end{align}
 
 By definition, $\tilde\beta_s = 0$ for $n_s \le 18d$, which indicates that 
 \begin{align}
     \Mean\bracket*{\abrace*{\tilde\beta_s,\mu_s - \hat\mu_s}^2\middle| n_\cdot}^{\nicefrac{1}{2}} = 0.
 \end{align}
 In the case of $n_s > 18d$, by utilizing the fact $\tilde\beta_s$ and $\hat\mu_s$ are independent conditioned on $n_\cdot$ due to the sample spilitting and $\tilde\beta_s \in \mathbb{S}_{d-1}$, we obtain
 \begin{align}
     \Mean\bracket*{\abrace*{\tilde\beta_s,\mu_s - \hat\mu_s}^2\middle| n_\cdot} \le& \sup_{v \in \mathbb{S}_{d-1}} \Mean\bracket*{\abrace*{v,\mu_s - \hat\mu_s}^2\middle| n_\cdot} = e^2_{\mathrm{mean},s}.
 \end{align}
 Consequently, we have
 \begin{align}
     \Mean\bracket*{\widehat{\norm*{\beta_\cdot}}^2\middle| n_\cdot}^{\nicefrac{1}{2}}\Mean\bracket*{\abrace*{\tilde\beta_s,\mu_s - \hat\mu_s}^2\middle| n_\cdot}^{\nicefrac{1}{2}} \le \paren*{B + e_{\mathrm{norm}}}e_{\mathrm{mean},s}\ind\cbrace*{n_s > 18d}. \label{eq:upper-d1}
 \end{align}

 (Second term in \cref{thm:decompose}) Recall the second term in \cref{thm:decompose}
 \begin{align}
     \sigma_X\Mean\bracket*{\paren*{\widehat{\norm{\beta_\cdot}} - \overline{\norm*{\beta^*_\cdot}}}^2\middle|n_\cdot}^{\nicefrac{1}{2}}.
 \end{align}
 Using the notation of $e_{\mathrm{norm}}$, we have
 \begin{align}
     \sigma_X\Mean\bracket*{\paren*{\widehat{\norm{\beta_\cdot}} - \overline{\norm*{\beta^*_\cdot}}}^2\middle|n_\cdot}^{\nicefrac{1}{2}} = \sigma_Xe_{\mathrm{norm}}. \label{eq:upper-d2}
 \end{align}

 (Third term in \cref{thm:decompose}) Recall the third term in \cref{thm:decompose}
 \begin{align}
     \sigma_X\overline{\norm{\beta^*_\cdot}}\Mean\bracket*{\norm*{\tilde\beta_s - \frac{\beta^*_s}{\norm{\beta^*_s}}}^2\middle|n_\cdot}^{\nicefrac{1}{2}}.
 \end{align}
 Substituting $e_{\mathrm{coef},s}$ yields
 \begin{align}
     \sigma_X\overline{\norm{\beta^*_\cdot}}\Mean\bracket*{\norm*{\tilde\beta_s - \frac{\beta^*_s}{\norm{\beta^*_s}}}^2\middle|n_\cdot}^{\nicefrac{1}{2}} = \sigma_X\overline{\norm{\beta^*_\cdot}}e_{\mathrm{coef},s}. \label{eq:upper-d3}
 \end{align}

 (Fourth term in \cref{thm:decompose}) Recall the fourth term in \cref{thm:decompose}
 \begin{align}
     \Mean\bracket*{\paren*{\sum_{s \in [M]}\hat{p}_{s}\abrace*{\hat\beta'_{s} - \beta^*_{s},\hat\mu'_{s}}}^2\middle|n_\cdot}^{\nicefrac{1}{2}}.
 \end{align}
  Due to the sample splitting, $\hat\beta'_s$ and $\hat\mu'_s$ are mutually independent. Also, we have $\Mean\bracket{\abrace{\hat\beta'_{s} - \beta^*_{s},\hat\mu'_{s}}}=0$. Hence,
  \begin{align}
      & \Mean\bracket*{\paren*{\sum_{s \in [M]}\hat{p}_{s}\abrace*{\hat\beta'_{s} - \beta^*_{s},\hat\mu'_{s}}}^2\middle|n_\cdot} \\
      =& \sum_{s \in [M]}\hat{p}^2_s\Mean\bracket*{\abrace*{\hat\beta'_{s} - \beta^*_{s},\hat\mu'_{s}}^2\middle|n_\cdot} \\
      =& \sum_{s \in [M]}\hat{p}^2_s\Mean\bracket*{\abrace*{\hat\beta'_{s} - \beta^*_{s},\hat\mu'_{s}}^2\middle|n_\cdot}\ind\cbrace*{n_s > 12d}.
  \end{align}
  Since $\Mean\bracket{(\hat\mu'_s)(\hat\mu'_s)^\top} = \mu\mu^\top + \frac{\sigma^2_X}{n'_{2,s}}I = \Sigma_s - \paren{1-\frac{1}{n'_{2,s}}}\sigma^2_XI$, we have
  \begin{align}
      & \Mean\bracket*{\paren*{\sum_{s \in [M]}\hat{p}_{s}\abrace*{\hat\beta'_{s} - \beta^*_{s},\hat\mu'_{s}}}^2\middle|n_\cdot} \\
      =& \sum_{s \in [M]}\hat{p}^2_s\Mean\bracket*{\norm*{\paren*{\Sigma_s - \paren*{1-\frac{1}{n'_{2,s}}}\sigma^2_XI}^{\nicefrac{1}{2}}\paren*{\hat\beta'_{s} - \beta^*_{s}}}^2\middle|n_\cdot}\ind\cbrace*{n_s > 12d} \\
      \le& \sum_{s \in [M]}\hat{p}^2_s\Mean\bracket*{\norm*{\Sigma_s^{\nicefrac{1}{2}}\paren*{\hat\beta'_{s} - \beta^*_{s}}}^2\middle|n_\cdot}\ind\cbrace*{n_s > 12d} = \sum_{s\in[M]}\hat{p}_s^2e^2_{\mathrm{coef}',s}\ind\cbrace*{n_s > 12d}, \label{eq:upper-d4}
  \end{align}
  where we use the fact $\Sigma_s \succeq \Sigma_s - \paren{1-\frac{1}{n'_{2,s}}}\sigma^2_XI$, and for symmetric matrices $A$ and $B$ such that $A \succeq B$, $\abrace{v,Av} \ge \abrace{v, Bv}$ for any $v \in \RealSet$. For $n_s \le 18d$, the error is zero because $\hat\mu'_s = 0$

  (Fifth term in \cref{thm:decompose}) Recall the fifth term in \cref{thm:decompose}
  \begin{align}
      \Mean\bracket*{\paren*{\sum_{s \in [M]}\hat{p}_{s}\abrace*{\beta^*_{s},\hat\mu'_{s} - \mu_{s}}}^2\middle|n_\cdot}^{\nicefrac{1}{2}}.
  \end{align}
  Since $\hat\mu'_s$ are independent, we have
  \begin{align}
      & \Mean\bracket*{\paren*{\sum_{s \in [M]}\hat{p}_{s}\abrace*{\beta^*_{s},\hat\mu'_{s} - \mu_{s}}}^2\middle|n_\cdot} \\
      =& \begin{multlined}[t][.9\textwidth]
          \sum_{s \in [M]}\hat{p}_s^2\Mean\bracket*{\abrace*{\beta^*_{s},\hat\mu'_{s} - \mu_{s}}^2\middle|n_\cdot}\ind\cbrace*{n_s > 12d} \\ + \sum_{s,s' \in [M]}\hat{p}_s\hat{p}_{s'}\ind\cbrace*{n_s < 12d, n_{s'} < 12d}\abrace*{\beta^*_s,\mu_s}\abrace*{\beta^*_{s'},\mu_{s'}} 
      \end{multlined} \\
      \le& \sum_{s \in [M]}\hat{p}_s^2\norm*{\beta^*_s}^2\sup_{v \in \mathbb{S}_{d-1}}\Mean\bracket*{\abrace*{v,\hat\mu'_{s} - \mu_{s}}^2\middle|n_\cdot}\ind\cbrace*{n_s > 12d} + \frac{144d^2M^2}{n^2} \\
      \le& \sum_{s \in [M]}\hat{p}_s^2B^2e^2_{\mathrm{mean}',s}\ind\cbrace*{n_s > 12d} + o\paren*{\frac{1}{n}}. \label{eq:upper-d5}
  \end{align}

  (Sixth term in \cref{thm:decompose}) Recall the sixth term in \cref{thm:decompose}
  \begin{align}
      \abs*{\sum_{s' \in [M]}\paren*{\hat{p}_{s'} - p_{s'}}\abrace*{\beta^*_{s'},\mu_{s'}}},
  \end{align}
  which is equivalent to $e_{\mathrm{prob}}$.

  By combining \cref{thm:decompose,eq:upper-d1,eq:upper-d2,eq:upper-d3,eq:upper-d4,eq:upper-d5}, we get
  \begin{multline}
     \Mean\bracket*{\paren*{\hat{f}_n(X,S) - f^*_{\mathrm{DP}}(X,S)}^2} \\
     \le \sum_{s \in [M]}p_s\Mean\bracket[\Bigg]{\paren[\bigg]{\paren*{B + e_{\mathrm{norm}}}e_{\mathrm{mean},s}\ind\cbrace*{n_s > 18d} + \sigma_Xe_{\mathrm{norm}} + \sigma_X\overline{\norm{\beta^*_\cdot}}e_{\mathrm{coef},s} \\ + \paren*{\sum_{s'\in[M]}\hat{p}_{s'}^2e^2_{\mathrm{coef}',s'}\ind\cbrace*{n_{s'} > 12d}}^{\nicefrac{1}{2}} \\ + \paren*{\sum_{s' \in [M]}\hat{p}_{s'}^2B^2e^2_{\mathrm{mean}',s'}\ind\cbrace*{n_{s'} > 12d} + o\paren*{\frac{1}{n}}}^{\nicefrac{1}{2}} + e_{\mathrm{prob}}}^2}.
  \end{multline}
 The triangle inequality gives that
  \begin{multline}
     \Mean\bracket*{\paren*{\hat{f}_n(X,S) - f^*_{\mathrm{DP}}(X,S)}^2} \\
     \le \sum_{s \in [M]}7p_s\paren[\bigg]{B^2\Mean\bracket*{e^2_{\mathrm{mean},s}\ind\cbrace*{n_s > 18d}} + \Mean\bracket*{e_{\mathrm{norm}}^2}\Mean\bracket*{e^2_{\mathrm{mean},s}\ind\cbrace*{n_s > 18d}} \\ + \sigma^2_X\Mean\bracket*{e^2_{\mathrm{norm}}} + \sigma_X^2\overline{\norm{\beta^*_\cdot}}^2\Mean\bracket*{e^2_{\mathrm{coef},s}}} + \sum_{s \in [M]}7\paren[\Bigg]{\Mean\bracket*{\hat{p}^2_se^2_{\mathrm{coef}',s}\ind\cbrace*{n_s > 12d}} \\ + \Mean\bracket*{\hat{p}_s^2B^2e^2_{\mathrm{mean}',s}\ind\cbrace*{n_s > 12d}}} + 7\Mean\bracket*{e^2_{\mathrm{prob}}} + o\paren*{\frac{1}{n}}.
  \end{multline}

  By applying \cref{thm:decomp-first,lem:inv-n}, we have
  \begin{align}
      \Mean\bracket*{e^2_{\mathrm{mean},s}\ind\cbrace*{n_s > 18d}} \le& \Mean\bracket*{\frac{6\sigma^2_X}{n_s}\ind\cbrace*{n_s > 18d}} \le \frac{6\sigma^2_X}{p_s(n+1)}\paren*{1 + \frac{1}{18d}}.
  \end{align}
  Also, from \cref{thm:decomp-second,lem:small-n,lem:diff-p}, we have
  \begin{align}
      \Mean\bracket*{e_{\mathrm{norm}}^2} \le& \Mean\bracket*{\paren*{\sqrt{\frac{189e^{10}\sigma^2_\xi Md}{\sigma^2_Xn} + \sum_{s\in[M]}\ind\cbrace*{n_s \le 18d}\frac{n_s}{n}} + \abs*{\sum_{s\in[M]}\paren*{\hat{p}_s - p_s}\norm*{\beta^*_s}}}^2} \\
      \le& \frac{378e^{10}\sigma^2_\xi Md}{\sigma^2_Xn} + \sum_{s\in[M]}\Mean\bracket*{\ind\cbrace*{n_s \le 18d}\frac{2n_s}{n}} + \frac{2}{n}\Var\paren*{\norm*{\beta^*_S}} \\
      \le& \frac{378e^{10}\sigma^2_\xi Md}{\sigma^2_Xn} + \sum_{s\in[M]}\p\cbrace*{n_s \le 18d}\frac{36d}{n} + \frac{2\max_s\norm*{\beta^*_s}^2}{n} \\
      \le& \frac{378e^{10}\sigma^2_\xi Md}{\sigma^2_Xn} + \sum_{s\in[M]}\frac{36d}{n}e^{-np_s/8} + \frac{2B^2}{n} \\
      =& \frac{378e^{10}\sigma^2_\xi Md}{\sigma^2_Xn} + \frac{2B^2}{n} + o\paren*{\frac{1}{n}},
  \end{align}
  provided that $n > 36d/\min_{s\in[M]}p_s$. By utilizing \cref{thm:decomp-forth,lem:inv-n,lem:small-n}, we have
  \begin{align}
      \Mean\bracket*{e^2_{\mathrm{coef},s}} \le& \Mean\bracket*{\frac{756e^{10}\sigma^2_\xi d}{\sigma^2_X\norm{\beta^*_s}^2n_{s}}\ind\cbrace*{n_s > 18d}} + \p\cbrace*{n_s \le 18d} \\
      \le& \frac{756e^{10}\sigma^2_\xi d}{p_s\sigma^2_X\norm{\beta^*_s}^2n}\paren*{1 + \frac{1}{18d}} + e^{-np_s/8} \\
      =& \frac{756e^{10}\sigma^2_\xi d}{p_s\sigma^2_X\norm{\beta^*_s}^2n}\paren*{1 + \frac{1}{18d}} + o\paren*{\frac{1}{n}},
  \end{align}
  provided that $n > 36d/\min_{s\in[M]}p_s$. Application of \cref{thm:decomp-fifth} gives
  \begin{align}
      \Mean\bracket*{\hat{p}^2_se^2_{\mathrm{coef}',s}\ind\cbrace*{n_s > 12d}}  \le& \Mean\bracket*{\paren*{\frac{n_s}{n}}^2\paren*{\frac{4\sigma^2_\xi}{n_{s}} + 540e^{10}\sigma^2_\xi\paren*{\frac{4d}{n_s}}^2}\ind\cbrace*{n_s > 12d}} \\
      \le& \Mean\bracket*{\frac{n_s}{n}\ind\cbrace*{n_s > 12d}}\frac{4\sigma^2_\xi}{n_s} + o\paren*{\frac{1}{n}} \le p_s\frac{4\sigma^2_\xi}{n_s} + o\paren*{\frac{1}{n}}.
  \end{align}
  By \cref{thm:decomp-sixth}, we have
  \begin{align}
      \Mean\bracket*{\hat{p}_s^2e^2_{\mathrm{mean}',s}\ind\cbrace*{n_s > 12d}} \le& \Mean\bracket*{\frac{n_s}{n}\frac{4\sigma^2_X}{n}\ind\cbrace*{n_s > 12d}} \le p_s\frac{4\sigma^2_X}{n}.
  \end{align}
  Application of \cref{lem:diff-p} with $a_s = \abrace{\beta^*_s,\mu_s}$ into $e^2_{\mathrm{prob}}$ yields
  \begin{align}
      \Mean\bracket*{e^2_{\mathrm{prob}}} \le& \frac{1}{n}\Var\paren*{\abrace{\beta^*_S,\mu_S}} \le \frac{B^2U^2}{n}. 
  \end{align}

  Synthesizing the results so far, there exists an universal constant $C > 0$ such that 
  \begin{multline}
     \Mean\bracket*{\paren*{\hat{f}_n(X,S) - f^*_{\mathrm{DP}}(X,S)}^2} \\
     \le \sum_{s \in [M]}Cp_s\paren*{\frac{\sigma^2_XB^2}{p_sn} + \frac{\sigma^2_\xi Md}{n} + \frac{\sigma^2_XB^2}{n} + \frac{\sigma^2_\xi \overline{\norm{\beta^*_\cdot}}^2 d}{\norm{\beta^*_s}^2p_sn}} \\ + \sum_{s \in [M]}C\paren*{\frac{p_s\sigma^2_X}{n} + \frac{p_s\sigma^2_X}{n}} + C\frac{B^2U^2}{n} + o\paren*{\frac{1}{n}}.
  \end{multline}
  Consequently, there exists an universal constant $C > 0$ such that
  \begin{multline}
      \Mean\bracket*{\paren*{\hat{f}_n(X,S) - f^*_{\mathrm{DP}}(X,S)}^2} \le \\ C\paren*{\frac{\sigma^2_XB^2M}{n} + \frac{\sigma^2_\xi Md}{n} + \frac{\sigma^2_XB^2}{n} + \frac{\sigma^2_\xi B^2Md}{n} + \frac{\sigma^2_X}{n} + \frac{\sigma^2_X}{n} + \frac{B^2U^2}{n}} + o\paren*{\frac{1}{n}}.
  \end{multline}
  Then, the dominating terms match the claim.
\end{proof}

\section{Details of Lower Bound Analyses}
This section provides the proofs of the lower bound analyses results.

\subsection{Proof of Lower Bound in \cref{thm:main-err}}
\begin{theorem}\label{thm:lower}
 If $M(d-1) > 16$, there exists an universal constant $C > 0$ such that for any $\alpha > 0$ and $\delta \in (0,1)$, 
 \begin{align}
     \mathcal{E}_n\paren*{\alpha,\delta} \ge C\frac{\sigma^2_\xi B^2dM}{n} - o\paren*{\frac{1}{n}}.
 \end{align}
\end{theorem}

\begin{proof}[Proof of \cref{thm:lower}]
 The Varshamov-Gilbert bound guarantees that there exists a subset $\dom{V}' \subseteq \dom{V}$ such that $\abs{\dom{V}'} \ge 2^{M(d-1)/8}$ and $d_H(v_s,v'_s) \ge (d-1)/8$ for any $v,v' \in \dom{V'}$. With the choice of $\epsilon^2_s = (\frac{d-1}{16} - \frac{1}{M})\sigma^2_\xi/2\sigma^2_XB_s^2n_s$, we confirm by \cref{thm:distinct-params1} that $\inf_\pi\frac{1}{K}\sum_{v \in \dom{V}'}\KL\paren*{\pi_{\theta_v|n_\cdot},\pi} \le \max_{v,v'\in\dom{V}'}\KL\paren*{\pi_{\theta_{v'}|n_\cdot},\pi} \le \ln(\abs{\dom{V'}}/4)/2 \ge M(d-1)/16 - 1$. From \cref{thm:distinct-params1} and the fact $d_H(v_s,v'_s) \ge (d-1)/8$, we can apply \cref{thm:applied-fano} with $\epsilon = \sum_{s\in[M]}p_s\frac{(\sum_{s'\in[M]}p_sB_s)^2}{B^2_s}(\frac{d-1}{16} - \frac{1}{M})\sigma^2_\xi/2n_s$. From the fact that $\Mean\bracket{\frac{1}{n_s+1}} = \frac{1}{p_s(n+1)}(1-(1-p_s)^n)$ due to \autocite{Chao1972NegativeVariables}, there exists an universal constant $C > 0$ such that $\Mean[\frac{\epsilon}{2}] \ge C\paren{\frac{1}{M}\sum_{s\in[M]}\frac{(\sum_{s'\in[M]}p_{s'}B_{s'})^2}{B^2_s}}\frac{\sigma^2_\xi Md}{n} - o(\frac{1}{n})$  We can get the claim by confirming that there exists $B_1,...,B_M$ such that $\paren{\frac{1}{M}\sum_{s\in[M]}\frac{(\sum_{s'\in[M]}p_{s'}B_{s'})^2}{B^2_s}} = B^2$ and $B_s \le B$. Because for $B_2 = ... = B_M = B$, tending $B_1$ to 0 results in $\paren{\frac{1}{M}\sum_{s\in[M]}\frac{(\sum_{s'\in[M]}p_{s'}B_{s'})^2}{B^2_s}}$ goes infinity, it is confirmed.
\end{proof}

\subsection{Proof of \crtcref{thm:applied-fano}}
\begin{proof}[Proof of \cref{thm:applied-fano}]
 Since the distribution of $n_\cdot$ is invariant against $\theta \in \Theta$, we have
 \begin{align}
     & \sup_{\theta \in \Theta}\Mean_\theta\bracket*{\mathcal{E}(\hat{f}_n;\theta)} \\
     =& \Mean\bracket*{ \sup_{\theta \in \Theta}\Mean_\theta\bracket*{\mathcal{E}(\hat{f}_n;\theta)\middle|n_\cdot}} \\
     \ge& \Mean\bracket*{ \max_{\theta \in \hat\Theta}\Mean_\theta\bracket*{\mathcal{E}(\hat{f}_n;\theta)\middle|n_\cdot}} \\
     \ge& \Mean\bracket*{ \frac{1}{\abs{\hat\Theta}}\sum_{\theta \in \hat\Theta}\Mean_\theta\bracket*{\mathcal{E}(\hat{f}_n;\theta)\middle|n_\cdot}}.
 \end{align}
 Given $\epsilon$ possibly dependent on $n_\cdot$, application of the Markov inequality yields
 \begin{align}
     & \sup_{\theta \in \Theta}\Mean_\theta\bracket*{\mathcal{E}(\hat{f}_n;\theta)} \\
     \ge& \Mean\bracket*{ \frac{\epsilon}{\abs{\hat\Theta}}\sum_{\theta \in \hat\Theta}\p_\theta\cbrace*{\mathcal{E}(\hat{f}_n;\theta) \ge \epsilon\middle|n_\cdot}}.
 \end{align}
 If $\inf_f \mathcal{E}(f;\theta)\lor\mathcal{E}(f;\theta') \ge \epsilon$ for any $\theta,\theta' \in \hat\Theta$, $\mathcal{E}(f;\theta) < \epsilon$ implies $\mathcal{E}(f;\theta') \ge \epsilon$ for any $\theta' \in \hat\Theta$ such that $\theta \ne \theta'$. Hence, there exists a partion $\cbrace{\dom{F}_\theta}_{\theta \in \hat\Theta}$ of all the measurable functions $f:\RealSet^d\times[M]\to\RealSet$ such that $\cbrace{f : \mathcal{E}(f;\theta) < \epsilon} \subseteq \dom{F}_\theta$ for all $\theta \in \hat\Theta$. Consequently, we have
 \begin{align}
     \sup_{\theta \in \Theta}\Mean_\theta\bracket*{\mathcal{E}(\hat{f}_n;\theta)} \ge \Mean\bracket*{\epsilon\paren*{1-\frac{1}{\abs{\hat\Theta}}\sum_{\theta \in \hat\Theta}\p_\theta\cbrace*{\hat{f}_n \in \dom{F}_\theta\middle|n_\cdot}}}.
 \end{align}
 Application of the Fano's inequality and data processing inequality yields the claim.
\end{proof}

\subsection{Proof of \crtcref{thm:two-point-lower}}
To prove \cref{thm:two-point-lower}, we show the following more tight lower bound.
\begin{theorem}\label{thm:two-point-lower-tight}
     Let $\theta$ and $\theta'$ be the parameters of the distributions such that $ \frac{1}{2\sigma^2_X}\norm*{\mu_s - \mu'_s}^2 \coloneqq d_s < 1$ for all $s \in [M]$. Then, we have
    \begin{multline}
        \inf_{f \in \dom{L}^2}\mathcal{E}(f;\theta)\lor\mathcal{E}(f;\theta') \ge \\ \sum_{s \in [M]}p_s\frac{e^{-d_s}}{4}\paren[\Bigg]{\sigma^2_X\norm*{\frac{\overline{\norm{\beta_\cdot}}\beta_s}{\norm{\beta_s}} - \frac{\overline{\norm{\beta'_\cdot}}\beta'_s}{\norm{\beta'_s}}}^2\paren*{1+\frac{\norm{\mu_s-\mu'_s}^2}{4\sigma^2_X}}^{1+\frac{d}{2}} \\ + \paren[\Bigg]{- \abrace*{\frac{\overline{\norm{\beta_\cdot}}\beta_s}{\norm{\beta_s}} + \frac{\overline{\norm{\beta'_\cdot}}\beta'_s}{\norm{\beta'_s}},\frac{\mu_s-\mu'_s}{2}} \\+ \sum_{s'\in[M]}p_{s'}\paren*{\abrace*{\beta_{s'}-\beta'_{s'},\bar\mu_{s'}} + \abrace*{\beta_{s'}+\beta'_{s'},\frac{\mu_{s'}-\mu'_{s'}}{2}}}}^2\paren*{1+\frac{\norm{\mu_s-\mu'_s}^2}{4\sigma^2_X}}^{\frac{d}{2}}}.
    \end{multline}
\end{theorem}
\cref{thm:two-point-lower-tight} immediately gives \cref{thm:two-point-lower}.

We utilize the sufficient condition for the constrained optimization problem over a Banach space. Let $Z$ be a Banach space. We say a function $f:Z\to\RealSet$ is \emph{Gateaux differentiable} if the limit $\lim_{\tau \to 0}\frac{f(z+\tau u)-f(z)}{\tau}$ exists for any open set $U \subseteq Z$, any $z \in U$, and any $u \in Z$. We denote the Gateaux derivative of $f$ at $z \in Z$, a linear mapping from $u \in Z$ to $\lim_{\tau \to 0}\frac{f(z+\tau u)-f(z)}{\tau}$, as $D_Gf(z)$. We abuse $0$ to denote the mapping that always outputs $0$. 
\begin{proof}[Proof of \cref{thm:two-point-lower-tight}]
  Let $q_s$ and $q'_s$ be the density function of $X_s$ with the parameters $(\beta_\cdot,\mu_\cdot)$ and $(\beta'_\cdot,\mu'_\cdot)$, respectively, regarding the base measure $\lambda$. Since $X_s$ follows the Gaussian distribution, we can choose $\lambda$ as the Lebesgue measure. Given $\eta \in [0,1]$, we have
  \begin{align}
      & \mathcal{E}(f;\beta_\cdot,\mu_\cdot)\lor\mathcal{E}(f;\beta'_\cdot,\mu'_\cdot) \\
      \ge& \eta\mathcal{E}(f;\beta_\cdot,\mu_\cdot) + (1-\eta)\mathcal{E}(f;\beta'_\cdot,\mu'_\cdot) \\
      =& \sum_{s \in [M]}p_s\int\paren*{\eta \paren*{f(x,s) - f_{\beta_\cdot,\mu_\cdot}(x,s)}^2q_s(x) + (1-\eta)\paren*{f(x,s) - f_{\beta'_\cdot,\mu'_\cdot}(x,s)}^2q'_s(x)}\lambda(dx).
  \end{align}
  Because $\eta\mathcal{E}(f;\beta_\cdot,\mu_\cdot) + (1-\eta)\mathcal{E}(f;\beta'_\cdot,\mu'_\cdot)$ is convex for $f$, and $\dom{L}^2$ is a Banach space, it is minimized if for any $u \in \dom{L}^2$, 
  \begin{align}
      \frac{d}{d\gamma}\paren*{\eta\mathcal{E}(f+\gamma u;\beta_\cdot,\mu_\cdot) + (1-\eta)\mathcal{E}(f+\gamma u;\beta'_\cdot,\mu'_\cdot)}\Bigr|_{\gamma=0} = 0.
  \end{align}
  
  The dominated convergence theorem gives
  \begin{align}
      & \frac{d}{d\gamma}\paren*{\eta\mathcal{E}(f+\gamma u;\beta_\cdot,\mu_\cdot) + (1-\eta)\mathcal{E}(f+\gamma u;\beta'_\cdot,\mu'_\cdot)}\Bigr|_{\gamma=0} \\
      =& \begin{multlined}[t][.95\textwidth]
        \sum_{s \in [M]}p_s\int \paren[\big]{\eta\paren*{f(x,s)-f_{\beta_\cdot,\mu_\cdot}(x,s)}u(x,s)q_s(x) \\ + (1-\eta)\paren*{f(x,s)-f_{\beta'_\cdot,\mu'_\cdot}(x,s)}u(x,s)q'_s(x)}\lambda(dx) 
        \end{multlined} \\
      =& \begin{multlined}[t][.95\textwidth]
        \sum_{s \in [M]}p_s\int \paren[\big]{f(x,s)\paren*{\eta q_s(x) + (1-\eta)q'_s(x)} \\ - \paren*{\eta f_{\beta_\cdot,\mu_\cdot}(x,s)q_s(x) + (1-\eta)f_{\beta'_\cdot,\mu'_\cdot}(x,s)q'_s(x)}}u(x,s)\lambda(dx).
        \end{multlined}
  \end{align}
  Consequently, $\eta\mathcal{E}(f;\beta_\cdot,\mu_\cdot) + (1-\eta)\mathcal{E}(f;\beta'_\cdot,\mu'_\cdot)$ is minimized at
  \begin{align}
      f(x,s) = \frac{\eta f_{\beta_\cdot,\mu_\cdot}(x,s)q_s(x) + (1-\eta)f_{\beta'_\cdot,\mu'_\cdot}(x,s)q'_s(x)}{\eta q_s(x) + (1-\eta)q'_s(x)}.
  \end{align}
  Hence, 
  \begin{align}
      & \eta\mathcal{E}(f+\gamma u;\beta_\cdot,\mu_\cdot) + (1-\eta)\mathcal{E}(f+\gamma u;\beta'_\cdot,\mu'_\cdot) \\
      \ge& \sum_{s \in [M]}p_s\int \frac{\eta(1-\eta)q_s(x)q'_s(x)}{\eta q_s(x)+(1-\eta)q'_s(x)}\paren*{f_{\beta_\cdot,\mu_\cdot}(x,s) - f_{\beta'_\cdot,\mu'_\cdot}(x,s)}^2\lambda(dx).
  \end{align}
  With $\eta = \nicefrac{1}{2}$, we have
  \begin{align}
      \frac{\eta(1-\eta)q_s(x)q'_s(x)}{\eta q_s(x)+(1-\eta)q'_s(x)} 
      =& \frac{1}{4}\frac{\sqrt{q_s(x)q'_s(x)}}{\frac{1}{2}\sqrt{\frac{q_s(x)}{q'_s(x)}}+\frac{1}{2}\sqrt{\frac{q'_s(x)}{q_s(x)}}} \\
      =& \frac{1}{4}\frac{\sqrt{q_s(x)q'_s(x)}}{\cosh\paren{\frac{1}{2}\ln\frac{q_s(x)}{q'_s(x)}}}.
  \end{align}
  Let $\bar\mu_s = \frac{1}{2}\paren{\mu_s+\mu'_s}$. Then, we have
  \begin{align}
    & \sqrt{q_s(x)q'_s(x)} \\
    =& \frac{1}{\sqrt{(2\pi)^d\sigma_X^{2d}}}\exp\paren*{-\frac{1}{4\sigma_X^2}\norm*{x - \mu_s}^2 - \frac{1}{4\sigma_X^2}\norm*{x - \mu'_s}^2} \\
    =& \begin{multlined}[t][.95\textwidth]\frac{1}{\sqrt{(2\pi)^d\sigma_X^{2d}}}\exp\paren[\Bigg]{-\frac{1}{2\sigma_X^2}\norm*{x - \bar\mu_s}^2 - \frac{1}{2\sigma_X^2}\norm*{\frac{\mu_s - \mu'_s}{2}}^2 \\ - \frac{1}{2\sigma_X^2}\paren*{\abrace*{x - \bar\mu_s,\frac{\mu'_s - \mu_s}{2} + \frac{\mu_s - \mu'_s}{2}}}} \end{multlined}\\
    =& \frac{1}{\sqrt{(2\pi)^d\sigma_X^{2d}}}\exp\paren*{-\frac{1}{2\sigma_X^2}\norm*{x - \bar\mu_s}^2 - \frac{1}{2\sigma_X^2}\norm*{\frac{\mu_s - \mu'_s}{2}}^2}.
  \end{align}
  Also, we have
  \begin{align}
    & \frac{q_s(x)}{q'_s(x)} \\
    =& \exp\paren*{-\frac{1}{2\sigma^2_X}\norm*{x - \mu_s}^2 + \frac{1}{2\sigma^2_X}\norm*{x - \mu'_s}^2 } \\
    =& \exp\paren*{-\frac{1}{\sigma^2_X}\abrace*{x - \bar\mu_s, \frac{\mu'_s - \mu_s}{2} - \frac{\mu_s - \mu'_s}{2}}} \\
    =& \exp\paren*{\frac{1}{\sigma^2_X}\abrace*{x - \bar\mu_s, \mu_s - \mu'_s}}.
  \end{align}
  Let $\bar{X}_s \sim N(\bar\mu_s,\sigma^2_XI)$. Then, we have
  \begin{align}
    & \mathcal{E}(f;\beta_\cdot,\mu_\cdot)\lor\mathcal{E}(f;\beta'_\cdot,\mu'_\cdot) \\
    \ge& \sum_{s \in [M]}p_se^{-\frac{1}{2\sigma^2_X}\norm{\frac{\mu_s-\mu'_s}{2}}^2}\Mean\bracket*{\frac{1}{4}\frac{\paren{f_{\beta_\cdot,\mu_\cdot}(\bar{X}_s,s) - f_{\beta'_\cdot,\mu'_\cdot}(\bar{X}_s,s)}^2}{\cosh\paren{\frac{1}{2\sigma^2_X}\abrace*{\bar{X}_s - \bar\mu_s, \mu_s - \mu'_s}}}} \\
    \ge& \sum_{s \in [M]}p_se^{-\frac{1}{2\sigma^2_X}\norm{\frac{\mu_s-\mu'_s}{2}}^2}\Mean\bracket*{\frac{1}{4}\frac{\paren{f_{\beta_\cdot,\mu_\cdot}(\bar{X}_s,s) - f_{\beta'_\cdot,\mu'_\cdot}(\bar{X}_s,s)}^2}{\cosh\paren{\frac{1}{2\sigma^2_X}\norm*{\bar{X}_s - \bar\mu_s}\norm*{\mu_s - \mu'_s}}}}, \label{eq:two-point-lower1}
  \end{align}
  where the last line is obtained from the Cauchy–Schwarz inequality.
    
  By definition, we have
  \begin{align}
    & f_{\beta_\cdot,\mu_\cdot}(x,s) - f_{\beta'_\cdot,\mu'_\cdot}(x,s) \\
    =& \begin{multlined}[t][.95\textwidth]
      \abrace*{\frac{\overline{\norm{\beta_\cdot}}\beta_s}{\norm{\beta_s}} - \frac{\overline{\norm{\beta'_\cdot}}\beta'_s}{\norm{\beta'_s}},x - \bar\mu_s} - \abrace*{\frac{\overline{\norm{\beta_\cdot}}\beta_s}{\norm{\beta_s}} + \frac{\overline{\norm{\beta'_\cdot}}\beta'_s}{\norm{\beta'_s}},\frac{\mu_s-\mu'_s}{2}} \\ + \sum_{s'\in[M]}p_{s'}\paren*{\abrace*{\beta_{s'}-\beta'_{s'},\bar\mu_{s'}} + \abrace*{\beta_{s'}+\beta'_{s'},\frac{\mu_{s'}-\mu'_{s'}}{2}}}.
    \end{multlined}
  \end{align}
  Conditioned on $\norm{\bar{X}_s - \bar\mu_s}=r$ for $r > 0$, $\bar{X}_s$ follows the uniform distribution over the $(d-1)$-sphere centered at $\bar\mu_s$. For a random variable $U$ uniformly distributed over the $(d-1)$-sphere centered at origin with the radius $r$, $\Mean[U] = 0$ and $\Mean[UU^\top]=\nicefrac{r^2}{d}I$. Hence, for a vector $v \in \RealSet^d$ and a scalar $c \in \RealSet$, we have
  \begin{align}
      & \Mean\bracket*{\paren*{\abrace*{v,\bar{X}_s-\mu_s} + c}^2 \middle| \norm{\bar{X}_s - \bar\mu_s}=r} \\
      =& \frac{r^2}{d}\norm*{v}^2 + c^2.
  \end{align}
  
  An elementary analysis yields that $\norm{\bar{X}_s - \bar\mu_s}^2 \sim \mathrm{Gamma}(\frac{d}{2},2\sigma^2_X)$, where $\mathrm{Gamma}(k,\theta)$ denotes the Gamma distribution with the shape parameter $k$ and scale parameter $\theta$. From the upper bound of the hyperbolic cosine as $\cosh(x) \le e^{\nicefrac{x^2}{2}}$, for a vector $v \in \RealSet^d$, scalars $c \in \RealSet$ and $c' > 0$, and a random variable $\gamma \sim \mathrm{Gamma}(k,\theta)$, we have
  \begin{align}
      &\Mean\bracket*{\frac{1}{\cosh(c'\sqrt{\gamma})}\paren*{\frac{\gamma}{d}\norm*{v}^2+c^2}} \\
      \ge& \Mean\bracket*{\paren*{\frac{\gamma}{d}\norm*{v}^2+c^2}e^{-\frac{c'^2\gamma}{2}}} \\
      =& \Mean\bracket*{\paren*{\frac{\gamma}{d}\norm*{v}^2+c^2}\sum_{m=0}^\infty(-1)^m\frac{c'^{2m}\gamma^m}{2^mm!}} \\
      =& \sum_{m=0}^\infty(-1)^m\frac{c'^{2m}}{2^mm!}\paren*{\frac{\norm{v}^2}{d}\theta^{m+1}\frac{\Gamma(k+m+1)}{\Gamma(k)}+c^2\theta^{m}\frac{\Gamma(k+m)}{\Gamma(k)}} \\
      =& \sum_{m=0}^\infty\paren*{-\frac{c'^{2}\theta}{2}}^m\frac{1}{m!}\paren*{\frac{k\norm{v}^2}{d}\theta\frac{\Gamma(k+1+m)}{\Gamma(k+1)} + c^2\frac{\Gamma(k+m)}{\Gamma(k)}} \\
      =& \frac{k\norm{v}^2}{d}\theta\paren*{1+\frac{c'^{2}\theta}{2}}^{k+1} + c^2\paren*{1+\frac{c'^{2}\theta}{2}}^{k},
  \end{align}
  where we use the fact that the hypergeometric function $_2F_1(a,b,b;z)=\sum_{m=0}^\infty\frac{\Gamma(a+m)}{\Gamma(a)}\frac{z^m}{m!} = (1-z)^{-a}$ for some $b$, provided $\abs{z} < 1$. By setting
  \begin{align}
      v =& \frac{\overline{\norm{\beta_\cdot}}\beta_s}{\norm{\beta_s}} - \frac{\overline{\norm{\beta'_\cdot}}\beta'_s}{\norm{\beta'_s}} \\
      c =& \begin{multlined}[t][.95\textwidth] - \abrace*{\frac{\overline{\norm{\beta_\cdot}}\beta_s}{\norm{\beta_s}} + \frac{\overline{\norm{\beta'_\cdot}}\beta'_s}{\norm{\beta'_s}},\frac{\mu_s-\mu'_s}{2}} \\ + \sum_{s'\in[M]}p_{s'}\paren*{\abrace*{\beta_{s'}-\beta'_{s'},\bar\mu_{s'}} + \abrace*{\beta_{s'}+\beta'_{s'},\frac{\mu_{s'}-\mu'_{s'}}{2}}} \end{multlined} \\
      c' =& \frac{1}{2\sigma^2_X}\norm*{\mu_s - \mu'_s} \\
      k =& \frac{d}{2}, \textand \theta = 2\sigma^2_X,
  \end{align}
  we have 
  \begin{align}
      & \Mean\bracket*{\frac{\paren{f_{\beta_\cdot,\mu_\cdot}(\bar{X}_s,s) - f_{\beta'_\cdot,\mu'_\cdot}(\bar{X}_s,s)}^2}{\cosh\paren{\frac{1}{2\sigma^2_X}\norm*{\bar{X}_s - \bar\mu_s}\norm*{\mu_s - \mu'_s}}}} \\
      =& \begin{multlined}[t][.95\textwidth]
        \sigma^2_X\norm*{\frac{\overline{\norm{\beta_\cdot}}\beta_s}{\norm{\beta_s}} - \frac{\overline{\norm{\beta'_\cdot}}\beta'_s}{\norm{\beta'_s}}}^2\paren*{1+\frac{\norm{\mu_s-\mu'_s}^2}{4\sigma^2_X}}^{1+\frac{d}{2}} + \paren[\Bigg]{- \abrace*{\frac{\overline{\norm{\beta_\cdot}}\beta_s}{\norm{\beta_s}} + \frac{\overline{\norm{\beta'_\cdot}}\beta'_s}{\norm{\beta'_s}},\frac{\mu_s-\mu'_s}{2}} \\+ \sum_{s'\in[M]}p_{s'}\paren*{\abrace*{\beta_{s'}-\beta'_{s'},\bar\mu_{s'}} + \abrace*{\beta_{s'}+\beta'_{s'},\frac{\mu_{s'}-\mu'_{s'}}{2}}}}^2\paren*{1+\frac{\norm{\mu_s-\mu'_s}^2}{4\sigma^2_X}}^{\frac{d}{2}}. \label{eq:two-point-lower2}
      \end{multlined}
  \end{align}
  Combining \cref{eq:two-point-lower1,eq:two-point-lower2} yields the claim.
\end{proof}

\subsection{Proof of \crtcref{thm:distinct-params1}}

\begin{proof}[Proof of \cref{thm:distinct-params1}]
 It is easy to check that $d_s = 0$, and for any $v,v' \in \dom{V}$, 
 \begin{align}
     &\norm*{\frac{\overline{\norm{\beta_{v,\cdot}}}\beta_{v,s}}{\norm{\beta_{v,s}}} - \frac{\overline{\norm{\beta_{v',\cdot}}}\beta_{v',s}}{\norm{\beta_{v',s}}}}^2 \\
     =& \paren*{\sum_{s'\in[M]}p_{s'}\norm*{\beta_{v,s}}}^2\norm*{\frac{\beta_{v,s}}{\norm{\beta_{v,s}}} - \frac{\beta_{v',s}}{\norm{\beta_{v',s}}}}^2 \\
     =& \paren*{\sum_{s'\in[M]}p_{s'}\norm*{\beta_{v,s}}}^2\paren*{\sum_{i \in [d-1]}\frac{\epsilon^2_s}{d-1}\paren*{v_{s,i} - v'_{s,i}}^2} \\
     =& 4\paren*{\sum_{s'\in[M]}p_{s'}B_s}^2\frac{\epsilon^2_s}{d-1}d_H(v_s,v'_s).
 \end{align}
 Since the density function of the Gaussian distribution is $L^2$ integrable, $\mathcal{E}(f;\theta) = \infty$ if $f$ is not $L^2$ integrable. Hence, $\inf_f\mathcal{E}(f;\theta_v)\lor\mathcal{E}(f;\theta_{v'}) = \inf_{f\in\dom{L}^2}\mathcal{E}(f;\theta_v)\lor\mathcal{E}(f;\theta_{v'})$, and we thus can apply \cref{thm:two-point-lower}. Then, we have
 \begin{align}
     \inf_f\mathcal{E}(f;\theta_v)\lor\mathcal{E}(f;\theta_{v'}) \ge \sum_{s\in[M]}p_s\paren*{\sum_{s'\in[M]}p_{s'}B_s}^2\frac{\sigma^2_X\epsilon^2_s}{d-1}d_H(v_s,v'_s)
 \end{align}

Conditioned on $n_\cdot$, the KL-divergence between $\pi_{\theta_v|n_\cdot}$ and $\pi_{\theta_{v'}|n_\cdot}$ is obtained as
\begin{align}
    \sum_{s \in [M]}n_s\paren*{\frac{1}{2\sigma^2_X}\norm*{\mu_{v,s}-\mu_{v',s}}^2 + \frac{\sigma^2_X}{2\sigma^2_\xi}\norm*{\beta_{v,s} - \beta_{v',s}}^2 + \frac{1}{2\sigma^2_\xi}\abrace*{\mu_{v,s},\beta_{v,s}-\beta_{v',s}}^2}.
\end{align}
 Hence, we have
 \begin{align}
     \KL\paren*{\pi_{\theta_v|n_\cdot},\pi_{\theta_{v'},n_\cdot}} = \sum_{s\in[M]}\frac{2\sigma^2_XB_s^2n_s\epsilon_s^2}{\sigma^2_\xi(d-1)}d_H(v_s,v'_s).
 \end{align}
\end{proof}

\end{document}